\documentclass[11pt]{amsart}

\usepackage{amssymb}
\usepackage{graphicx}   
\usepackage{mathrsfs}
\usepackage{amsmath}
\usepackage{amsthm}
\usepackage{amsfonts}
\usepackage{latexsym}

\newtheorem{thm}{Theorem}[section]
\newtheorem{prop}[thm]{Proposition}
\newtheorem{lem}[thm]{Lemma}
\newtheorem{cor}[thm]{Corollary}

\theoremstyle{definition}

\theoremstyle{claim}
\newtheorem{claim}[thm]{Claim}

\theoremstyle{remark}
\newtheorem{remark}[thm]{Remark}

\numberwithin{equation}{section}

\begin{document}


\title{$p$-Harmonic Maps to $S^1$ and Stationary Varifolds of Codimension 2}

\author{Daniel L. Stern}
\address{Department of Mathematics, Princeton University, 
Princeton, NJ 08544}
\email{dls6@math.princeton.edu}


\begin{abstract} We study the asymptotics as $p\uparrow 2$ of stationary $p$-harmonic maps $u_p\in W^{1,p}(M,S^1)$ from a compact manifold $M^n$ to $S^1$, satisfying the natural energy growth condition
$$\int_M|du_p|^p=O(\frac{1}{2-p}).$$
Along a subsequence $p_j\to 2$, we show that the singular sets $Sing(u_{p_j})$ converge to the support of a stationary, rectifiable $(n-2)$-varifold $V$ of density $\Theta_{n-2}(\|V\|,\cdot)\geq 2\pi$, given by the concentrated part of the measure
$$\mu=\lim_{j\to\infty}(2-p_j)|du_{p_j}|^{p_j}dv_g.$$
When $n=2$, we show moreover that the density of $\|V\|$ takes values in $2\pi\mathbb{N}$. Finally, on every compact manifold of dimension $n\geq 2$ we produce examples of nontrivial families $(1,2)\ni p\mapsto u_p\in W^{1,p}(M,S^1)$ of such maps via natural min-max constructions.
\end{abstract}

\maketitle



\section{Introduction}

\hspace{3mm} In their 1995 paper \cite{HL2}, Hardt and Lin consider the following question: given a simply connected domain $\Omega\subset \mathbb{R}^2$ and a map $g:\partial\Omega\to S^1$ of nonzero degree, what can be said about the limiting behavior of maps 
$$u_p\in W_g^{1,p}(\Omega,S^1):=\{u\in W^{1,p}(\Omega,S^1)\mid u_p|_{\partial\Omega}=g\}$$ 
minimizing the $p$-energy
$$\int_{\Omega}|du_p|^p=\min\{\int_{\Omega}|du|^p\mid u\in W_g^{1,p}(\Omega,S^1)\}$$
as $p\in (1,2)$ approaches $2$ from below? They succeed in showing--among other things--that away from a collection $A$ of $|deg(g)|$ singularities, a subsequence $u_{p_j}$ converges strongly to a harmonic map $v\in C^1_{loc}(\Omega\setminus A,S^1)$, and the measures 
$$\mu_j=(2-p_j)|du_j|^{p_j}(z)dz$$
converge to the sum $\Sigma_{a\in A}2\pi \delta_a$ of Dirac masses on $A$ \cite{HL2}. Moreover, the singular set $A=\{a_1,\ldots, a_{|\deg(g)|}\}$ minimizes a certain ``renormalized energy" function $W_g:\Omega^{|\deg(g)|}\to [0,\infty]$ associated to $g$, providing a strong constraint on the location of the singularities. In particular, though the homotopically nontrivial boundary map $g$ admits no extension to an $S^1$-valued map of finite Dirichlet energy--i.e, $W_g^{1,2}(\Omega,S^1)=\varnothing$--the limit of the $p$-energy minimizers as $p\uparrow 2$ provides us with a natural candidate for the optimal harmonic extension of $g$ to an $S^1$-valued map on $\Omega$.

\hspace{3mm} The results of \cite{HL2} were inspired in large part by the similar results of Bethuel, Brezis, and H\'{e}lein--contained in the influential monograph \cite{BBH}--concerning the asymptotics for minimizers $u_{\epsilon}$ of the Ginzburg-Landau functionals
$$E_{\epsilon}:W^{1,2}(\Omega,\mathbb{R}^2)\to \mathbb{R},\text{ }E_{\epsilon}(u)=\int_{\Omega}\frac{1}{2}|du|^2+\frac{(1-|u|^2)^2}{4\epsilon^2}$$
as $\epsilon\to 0$, with the measures
$$\mu_{\epsilon}:=\frac{|du_{\epsilon}(z)|^2}{2|\log\epsilon|}dz$$
taking on the role played by the measures $(2-p)|du|^p(z)dz$ in the setting of \cite{HL2}. In recent decades, the asymptotics for critical points of the Ginzburg-Landau functionals $E_{\epsilon}$ in higher dimensions have also been studied by a number of authors, often with an emphasis on the relationship betweeen concentration phenomena for the measures $\mu_{\epsilon}$ and minimal submanifolds of codimension two (see, for instance, \cite{BBO},\cite{BOSp},\cite{Cheng},\cite{LR},\cite{Ste1},\cite{Ste2}, among many others). A typical result says roughly that if the measures $\mu_{\epsilon}$ have uniformly bounded mass, then a subsequence $(C^0)^*$-converges as $\epsilon\to 0$ to a limiting measure that decomposes into two pieces: a concentrated component given by a stationary, rectifiable varifold $V^{GL}$ of codimension two, and a diffuse measure of the form $|h|^2dvol$ for some harmonic one-form $h$ (which vanishes under mild compactness assumptions) (see, e.g., \cite{BBO},\cite{BOSp},\cite{Cheng},\cite{Ste2}). 

\hspace{3mm} Results of this type point to the possibility of employing variational methods for the Ginzburg-Landau functionals to produce minimal submanifolds of codimension two, but for the complete success of such efforts, we need an improved understanding of the concentration of $\mu_{\epsilon}$. In particular, the question of integrality (up to a factor of $\pi$) of the limiting varifold $V^{GL}$ has been resolved only in dimension two \cite{CM} and for local minimizers in higher dimensions \cite{LR}. If, on the other hand, one could establish integrality of $V^{GL}$ for general families of critical points (or families with bounded index), then the min-max methods of \cite{Ste1} and \cite{Ste2} would provide a new proof of the existence of nontrivial stationary, integral varifolds of codimension two, in the spirit of Guaraco's work for the Allen-Cahn equation in codimension one \cite{Gu}.

\hspace{3mm} In this paper, motivated by analogy with the Ginzburg-Landau setting, we investigate the limiting behavior as $p\uparrow 2$ of stationary $p$-harmonic maps $u_p\in W^{1,p}(M^n,S^1)$ from an arbitrary compact, oriented manifold $M$ to the circle. At the global level, we find that the limiting behavior of the maps $u_p$ and their energy measures strongly resembles the asymptotics described above for solutions to the Ginzburg-Landau equations. At the smallest scales, however, the comparatively straightforward blow-up analysis for $p$-harmonic maps leads us to some simpler arguments and sharper estimates than are currently available in the Ginzburg-Landau setting. 

\hspace{3mm} Our first result tells us that the limiting behavior of the energy measures
$$\mu_p=(2-p)|du_p|^pdvol_g$$
mirrors that of the measures $\mu_{\epsilon}$ in the Ginzburg-Landau setting, with the modest improvement of a sharp lower bound for the density of the concentration varifold:

\begin{thm}\label{mainthm1} Let $p_i\in (1,2)$ be a sequence with $\lim_{i\to\infty}p_i=2$, and let $u_i\in W^{1,p_i}(M^n,S^1)$ be a sequence of stationary $p_i$-harmonic maps from a compact, oriented Riemannian manifold $M^n$ to the circle, satisfying
\begin{equation}
\sup_i(2-p_i)\int_M|du_i|^{p_i}<\infty.
\end{equation}
Then (a subsequence of) the energy measures $\mu_i=(2-p_i)|du_i|^{p_i}dv_g$ converge weakly in $(C^0)^*$ to a limiting measure $\mu$ of the form
\begin{equation}
\mu=\|V\|+|\bar{h}|^2dvol_g,
\end{equation}
where $\bar{h}$ is a harmonic one-form, and $V$ is a stationary, rectifiable $(n-2)$ varifold. Furthermore, the support of $V$ is given by the Hausdorff limit
$$spt(V)=\lim_{i\to\infty}Sing(u_i)$$
of the singular sets $Sing(u_i)$, and the density $\Theta_{n-2}(\|V\|,\cdot)$ satisfies 
\begin{equation}
\Theta_{n-2}(\|V\|,x)\geq 2\pi\text{ for }x\in spt(V).
\end{equation}
\end{thm} 

In the course of proving Theorem \ref{mainthm1}, we also establish the following compactness result for the maps:

\begin{thm}\label{mapconvthm} Suppose that, in addition to the hypotheses of Theorem \ref{mainthm1}, either $b_1(M)=0$ or 
$$\sup_i \|du_i\|_{L^1(M)}<\infty.$$
Then (a subsequence of) the maps $u_i$ converge weakly in $W^{1,q}(M)$ for all $q\in (1,2)$, and strongly in $W^{1,2}_{loc}(M\setminus spt(V))$, to a limiting map $v\in \bigcap_{q\in [1,2)}W^{1,q}(M,S^1)$ that is harmonic away from $spt(V)$.
\end{thm}

\hspace{3mm} Next, we show that the concentrated measure $\|V\|$ is quantized in the two-dimensional setting, with a proof that's somewhat simpler than its analog \cite{CM} in the Ginzburg-Landau setting:

\begin{thm}\label{2dintthm} In the situation of Theorem \ref{mainthm1}, if $dim(M)=2$, then the density of the concentration varifold $V$ has the form
$$\|V\|=\Sigma_{x\in spt(V)}2\pi m_x\delta_x$$
for some $m_x\in \mathbb{N}$.
\end{thm}

As in the Ginzburg-Landau setting, the question of the integrality of $\frac{1}{2\pi}V$ remains open in higher dimensions, but we suspect that the answer will be affirmative.

\hspace{3mm} Finally, we employ min-max arguments similar to those in \cite{Ste1}, together with results of Wang on generalized Ginzburg-Landau functionals \cite{Wang}, to demonstrate the existence nontrivial families satisfying the hypotheses of Theorem \ref{mainthm1} on any compact manifold:

\begin{thm}\label{existthm} On every compact Riemannian manifold $M^n$ of dimension $n\geq 2$, there exists a family $(1,2)\ni p\mapsto u_p\in W^{1,p}(M,S^1)$ of stationary $p$-harmonic maps to $S^1$ for which
\begin{equation}
0<\liminf_{p\to 2}(2-p)E_p(u)\leq \limsup_{p\to 2}(2-p)E_p(u)<\infty.
\end{equation}
\end{thm}

\subsection{Outline of the Paper:} In Section \ref{sec2}, we review important facts about the structure of maps in $W^{1,p}(M,S^1)$, $p$-harmonic functions and weakly $p$-harmonic maps to $S^1$, and stationary $p$-harmonic maps. 

In Section \ref{bigests}, we record a sharp lower bound for the $p$-energy density of a stationary $p$-harmonic map $u\in W^{1,p}(M,S^1)$ on its singular set $Sing(u)$--a simpler and sharper analog of the $\eta$-ellipticity result (see \cite{BBO},\cite{LR}) for solutions of the Ginzburg-Landau equations. We then use this to obtain $p$-independent estimates for the $(n-2)$-current $T(u)$ encoding the topological singularities of $u$, in the dual Sobolev norms $W^{-1,q}=(W^{1,q'})^*$ for $q\in (1,p)$. 

In Section \ref{hodgeests}, we employ the results of the preceding sections to estimate separately the components of the Hodge decomposition of the one-form $ju=u^*(d\theta)$, first globally in $L^q$ for $q<p$, then in stronger norms away from $Sing(u)$. 

In Section \ref{sec5}, we use these estimates, together with some standard techniques from the study of energy concentration phenomena, to complete the proofs of Theorems \ref{mainthm1} and \ref{mapconvthm}. 

In Section \ref{sec6}, we prove Theorem \ref{2dintthm}, first under some compactness assumptions, using Theorem \ref{mapconvthm} and a Pohozaev-type identity, and then for the general case, by showing that the compactness assumptions hold at scales outside of which the normalized energy measures vanish. 

In Section \ref{sec7}, we employ min-max arguments like those in \cite{Ste1} together with Wang's results for generalized Ginzburg-Landau functionals \cite{Wang} to prove Theorem \ref{existthm}. We also include a short appendix, containing the proofs of some estimates which are of use to us, but do not play a central role in the paper.

\section*{Acknowledgements} I would like to thank my advisor Fernando Cod\'{a} Marques for his support and encouragement, and Yu Wang for his interest in this work. The author is partially supported by NSF grants DMS-1502424 and DMS-1509027.

\section{Preliminaries: The Structure of $W^{1,p}(M,S^1)$ and Circle-Valued $p$-Harmonic Maps}\label{sec2}

\subsection{Topological Singularities and Lifting in $W^{1,p}(M,S^1)$} 

\hfill

\hspace{3mm} Let $M^n$ be a compact, oriented Riemannian manifold, and consider the space $W^{1,p}(M,S^1)$ of circle-valued Sobolev maps, realized as the collection of complex-valued maps $u\in W^{1,p}(M,\mathbb{C})$ satisfying $|u|=1$ almost everywhere in $M$. For each $u\in W^{1,p}(M,S^1)$, we denote by $ju$ the one-form
\begin{equation}
ju:=u^*(d\theta)=u^1du^2-u^2du^1.
\end{equation}
Observe that $|du|=|ju|$ almost everywhere on $M$, so that $ju$ belongs to $L^p$. When $u$ is smooth, the form $ju$ is obviously closed, and it is a straightforward consequence of the Poincar\'{e} Lemma that $u$ has a local lifting of the form $u=e^{i\varphi}$ for some smooth, real-valued $\varphi$.  

\hspace{3mm} For general $u\in W^{1,p}(M,S^1)$, the exterior derivative $d[ju]$ is no longer well-defined pointwise, but since $ju$ belongs to $L^p$, we can still make sense of $d[ju]$ as a distribution in $W^{-1,p}$. Namely, one defines the \emph{distributional Jacobian} $T(u)$ of $u$ to be the $(n-2)$-current acting on smooth $(n-2)$-forms $\zeta\in \Omega^{n-2}(M)$ by
\begin{equation}
\langle T(u),\zeta\rangle:=\int_M ju\wedge d\zeta.
\end{equation}

\hspace{3mm} The analytic and measure-theoretic properties of distributional Jacobians for $S^1$-valued maps (and their analog for sphere-valued maps more generally) have been studied by a number of authors; we make no attempt to survey the many contributions here, but refer the reader to the papers \cite{ABO}, \cite{JS2}, \cite{Mir}, and the references therein for a sample. Note that for smooth, complex-valued maps, we have the pointwise relation
$$d(u^1du^2-u^2du^1)=2du^1\wedge du^2,$$
and since $u\mapsto du^1\wedge du^2$ defines a continuous map from $W^{1,2}(M,\mathbb{C})$ to the space of $L^1$ two-forms, it follows that
$$T(u)=2du^1\wedge du^2$$
holds for all $u\in W^{1,2}(M,S^1)$. In particular, since $rank(du)\leq 1$ almost everywhere, one deduces that $T(u)=0$ for all $u\in W^{1,2}(M,S^1)$. On the other hand, for $p\in [1,2)$, and $k\in \mathbb{Z}$, the maps $v_k: D_1^2\to S^1$ given by
$$v_k(z):=(z/|z|)^k$$
evidently lie in $W^{1,p}(D,S^1)$, with nontrivial distributional Jacobian
$$T(v_k)=2\pi k\cdot \delta_0.$$

\hspace{3mm} Observe now that if $u$ has the form $u=e^{i\varphi}$ for some real-valued $\varphi \in W^{1,p}$, then $T(u)$ is given by 
$$\langle T(u),\zeta\rangle=\int_M d\varphi\wedge d\zeta,$$
and since $\varphi$ can be approximated in $W^{1,p}(M,\mathbb{R})$ by smooth functions, it follows that $T(u)=0$. The following result of Demengel provides a useful converse--if the topological singularity vanishes, then $u$ lifts locally to a real-valued function in the same Sobolev space:

\begin{prop}\label{liftcrit}\emph{(\cite{Dem})} If $u\in W^{1,p}(B^n,S^1)$ and $T(u)=0$ in the ball $B^n$, then $u=e^{i\varphi}$ on $B^n$ for some $\varphi \in W^{1,p}(B^n,\mathbb{R})$.
\end{prop}
The significance of the lifting result for variational problems on $W^{1,p}(M,S^1)$ is clear: away from the support of the $(n-2)$-current $T(u)$, an $S^1$-valued solution $u$ of some geometric p.d.e. lifts locally to a function $\varphi$ solving an associated \emph{scalar} problem, for which a stronger regularity theory is often available.

\subsection{Weakly $p$-Harmonic Maps to $S^1$}

\hfill

\hspace{3mm} A map $u\in W^{1,p}(M,S^1)$ for $p\in (1,\infty)$ is called \emph{weakly $p$-harmonic} if it satisfies 
\begin{equation}\label{wkpharm1}
\int |du|^{p-2}\langle du,dv\rangle=\int |du|^p\langle u,v\rangle
\end{equation}
for all $v\in (W^{1,p}\cap L^{\infty})(M,\mathbb{R}^2)$. Writing $v=\varphi u+i\psi u$ in (\ref{wkpharm1}), it's easy to see that (\ref{wkpharm1}) holds if and only if
\begin{equation}\label{wkpharm2}
\int |du|^{p-2}\langle ju,d\psi\rangle=0
\end{equation}
for all $\psi \in W^{1,p}(M,\mathbb{R})$--i.e., when $ju$ satisfies
\begin{equation}\label{weakpharm3}
div(|ju|^{p-2}ju)=0
\end{equation}
distributionally on $M$. From (\ref{wkpharm2}), it is not hard to see that $u$ is weakly $p$-harmonic precisely when $u$ minimizes the $p$-energy $E_p(u)=\int_M|du|^p$ among all competitors of the form $e^{i\varphi}u$.

\hspace{3mm} In view of (\ref{weakpharm3}), wherever $u$ admits a local lifting $u=e^{i\varphi}$ for some real-valued $\varphi \in W^{1,p}$, we see that $u$ is weakly $p$-harmonic if and only if $\varphi$ is a $p$-harmonic function--i.e., a weak solution of 
$$div(|d\varphi|^{p-2}d\varphi)=0.$$
For $p\in (1,2)$, the $C^{1,\alpha}$ regularity of $p$-harmonic functions was established by DiBenedetto \cite{DiB} and Lewis \cite{Lew}. It is, moreover, possible to check that the H\"{o}lder exponent $\alpha$ and other relevant constants in the central estimates of \cite{DiB} and \cite{Lew} can be taken independent of $p$ for $p$ bounded away from $1$ and $\infty$. Rather than using the full strength of the $C^{1,\alpha}$ regularity, we will employ in this paper the following simpler estimates, whose proof we sketch in the appendix:

\begin{prop}\label{pharmfcnreg} Let $B_{2r}(x)$ be a geodesic ball in some manifold $M^n$ with $|sec(M)|\leq K$. If $\varphi\in W^{1,p}(B_{2r}(x),\mathbb{R})$ is a $p$-harmonic function for $p\in [\frac{3}{2},2]$, then for some constant $C(n,K)<\infty$, we have that
\begin{equation}
\|d\varphi\|_{L^{\infty}(B_r(x))}^p\leq C r^{-n}\|d\varphi\|^p_{L^p(B_{2r}(x))}
\end{equation}
and
\begin{equation}
r^{p}\|Hess(\varphi)\|_{L^p(B_r(x))}^p\leq C\|d\varphi\|_{L^p(B_{2r}(x))}^p
\end{equation}
\end{prop}

Combining this with the lifting criterion of Proposition \ref{liftcrit}, one obtains the following partial regularity result for weakly $p$-harmonic maps to the circle:

\begin{cor}\label{wkpharmreg} Let $p\in [\frac{3}{2},2]$, and let $B_{2r}(x)$ be a geodesic ball on a manifold $M^n$ with $|sec(M)|\leq K$. If $u\in W^{1,p}(B_{2r}(x),S^1)$ is a weakly $p$-harmonic map with vanishing distributional Jacobian
$$T(u)=0\text{ in }B_{2r}(x),$$
then 
\begin{equation}
\|du\|_{L^{\infty}(B_r(x))}^p\leq Cr^{-n}\|du\|_{L^p(B_{2r}(x))}^p
\end{equation}
and
\begin{equation}
r^p\|\nabla du\|_{L^p(B_r(x))}^p\leq C\|du\|_{L^p(B_{2r}(x))}^p.
\end{equation}
\end{cor}

\begin{remark} Though Corollary \ref{wkpharmreg} shows that weakly $p$-harmonic maps $u\in W^{1,p}(M,S^1)$ are reasonably smooth (with effective estimates) away from the support of $T(u)$, observe that the weak $p$-harmonic condition alone gives \emph{no} constraint on $T(u)$ itself. Indeed, given any $v\in W^{1,p}(M,S^1)$, we can minimize $\int_M|du|^p$ among all maps of the form $u=e^{i\varphi}v$ to find a weakly $p$-harmonic $u$ with topological singularity
$$T(u)=d[jv+d\varphi]=djv=T(v)$$
equal to that of $v$. The problem of minimizing $p$-energy among $S^1$-valued maps with prescribed singularities in $\mathbb{R}^2$--and, more generally, among $S^{k-1}$-valued maps with prescribed singularities in $\mathbb{R}^k$--is studied in detail in \cite{CH}.
\end{remark}

\subsection{$p$-Stationarity and Consequences}\label{pstatcons}

\hfill

\hspace{3mm} A map $u\in W^{1,p}(M,S^1)$ is said to be \emph{$p$-stationary}, or simply \emph{stationary}, if it is critical for the energy $E_p(u)$ with respect to perturbations of the form $u_t=u\circ\Phi_t$ for smooth families $\Phi_t$ of diffeomorphisms on $M$. Equivalently, $u$ is $p$-stationary if it satisfies the inner-variation equation
\begin{equation}\label{stat}
\int_M |du|^pdiv(X)-p|du|^{p-2}\langle du^*du,\nabla X\rangle=0
\end{equation}
for every smooth, compactly supported vector field $X$ on $M$. 

\hspace{3mm} The most-studied class of stationary $p$-harmonic maps (for arbitrary target manifolds) are the $p$-energy minimizers, whose regularity theory for $p\neq 2$ was first investigated in \cite{HL1} and \cite{Luck}. On the other hand, as we discuss in Section \ref{sec7}, one can also combine the results of \cite{Wang} with various min-max constructions to produce many examples non-minimizing stationary $p$-harmonic maps for certain non-integer values of $p$. 

\hspace{3mm} Given a stationary $p$-harmonic map $u\in W^{1,p}(M^n,S^1)$, for each geodesic ball $B_r(x)\subset M$, we define the $p$-energy density
\begin{equation}
\theta_p(u,x,r):=r^{p-n}\int_{B_r(x)}|du|^p.
\end{equation}
By standard arguments, it follows from the stationary equation (\ref{stat}) that the density $\theta_p(u,x,r)$ is nearly monotonic in $r$: Namely, taking $X$ in (\ref{stat}) of the form
$$X=\psi\frac{1}{2}\nabla dist(x,\cdot)^2$$
for some functions $\psi\in C_c^{\infty}(B_r(x))$ approximating the characteristic function $\chi_{B_r(x)}$, and employing the Hessian comparison theorem to estimate the difference $\nabla X-I$ in $B_r(x)$, one obtains the following well-known estimate (see, e.g., \cite{HL1}, sections 4 and 7):

\begin{lem}\label{mono} Let $u\in W^{1,p}(M,S^1)$ be a stationary $p$-harmonic map on a manifold $M^n$ with $|sec(M)|\leq K$. Then there is a constant $C(n,K)$ such that for any $x\in M$ and almost every $0<r<inj(M)$, we have the inequality
\begin{equation}\label{monoform}
\frac{d}{dr}[e^{Cr^2}\theta_p(u,x,r)]\geq pe^{Cr^2}r^{p-n}\int_{\partial B_r(x)}|du|^{p-2}|\frac{\partial u}{\partial \nu}|^2.
\end{equation}
In particular, $e^{Cr^2}\theta_p(u,x,r)$ is monotone increasing in $r$.
\end{lem}

\hspace{3mm} In light of the monotonicity result, it makes sense to define the pointwise energy density 
\begin{equation}
\theta_p(u,x):=\lim_{r\to 0}\theta_p(u,x,r).
\end{equation}
Perhaps the most significant consequence of Lemma \ref{mono} is the boundedness of blow-up sequences: Given a sequence of radii $inj(M)>r_j\to 0$, observe that the maps $u_j=u_{x,r_j}\in W^{1,p}(B_1^n(0),S^1)$ defined by
$$u_j(y):=u(\exp_x(r_jy))$$
are stationary $p$-harmonic with respect to the blown-up metrics
$$g_j(y):=r_j^{-2}([\exp_x]^*g)(r_jy)$$
on $B_1^n(0)$, with $p$-energy given by
$$E_p(u_j,B_1,g_j)=\int_{B_1(0)}|du_j|_{g_j}^pdv_{g_j}=r_j^{p-n}\int_{B_{r_j}(x)}|du|_g^pdv_g=\theta_p(u,x,r_j),$$
so it follows from Lemma \ref{mono} that the $p$-energies $E_p(u_j,B_1,g_j)$ are uniformly bounded from above as $r_j\to 0$. 

\hspace{3mm} For a local minimizer $u$ of the $p$-energy, one could then appeal to the compactness results of (\cite{HL1}, Section 4) to conclude immediately that a subsequence $u_{j_k}$ of such a blow-up sequence converges strongly to a minimizing tangent map $u_{\infty}$. For $p\in (1,2)$--the range of interest to us--it turns out that we can still obtain a strong convergence result without the minimizing assumption, but this relies on the following subtler result of \cite{NVV}:

\begin{prop}\label{strongcomp}\emph{(\cite{NVV}, Lemma 3.17)} Let $N$ be a compact homogeneous space with  left-invariant metric. For fixed $p\in (1,\infty)\setminus \mathbb{N}$, let $u_j\in W^{1,p}(B_2^n,N)$ be a sequence of maps which are stationary $p$-harmonic with respect to a $C^2$-convergent sequence of metrics $g_j\to g_{\infty}$ on $B_2$. If $\{u_j\}$ is uniformly bounded in $W^{1,p}(B_2,N)$, then some subsequence $u_{j_k}$ converges strongly in $W^{1,p}(B_1,N)$ to a map $u_{\infty}$ that is stationary $p$-harmonic with respect to $g_{\infty}$.
\end{prop}

\begin{remark} The significance of the condition $p\notin \mathbb{N}$ is that the $p$-energy has no conformally invariant dimension in this case, so that no bubbling can occur, and the proposition follows from arguments generalizing those of \cite{Lin} to the case $p\neq 2$ (see \cite{NVV}). The requirement that $N$ be a homogeneous space is a technical one, arising from the fact that, at present, the most general $\epsilon$-regularity theorem available for stationary $p$-harmonic maps (when $p\neq 2$) is that of \cite{TW} for homogeneous targets. It may be of interest to note that $\epsilon$-regularity (and consequently Proposition \ref{strongcomp}) holds for arbitrary compact targets $N$ for those stationary $p$-harmonic maps $u:M\to N$ constructed from critical points of generalized Ginzburg-Landau functionals, by virtue of Lemma 2.3 of \cite{Wang}.
\end{remark}

\hspace{3mm} Since $S^1$ is in any case a compact homogeneous space, the result of Proposition \ref{strongcomp} applies to stationary $p$-harmonic maps to $S^1$ for any $p\in (1,2)$, the range of interest. In particular, it follows that for any blow-up sequence $u_j=u_{x,r_j}\in W^{1,p}(B_2^n(0),S^1)$, $r_j\to 0$, we can extract a subsequence $r_{j_k}\to 0$ such that the maps $u_{j_k}$ converge strongly in $W^{1,p}(B_1^n(0),S^1)$ to a map $u_{\infty}\in W^{1,p}(B_1^n(0),S^1)$ which is stationary $p$-harmonic with respect to the flat metric, and satisfies
$$\theta_p(u_{\infty},0,r)=\theta_p(u,x)$$
for every $r>0$. Following standard arguments (see, e.g., \cite{HL1}), we can then apply the Euclidean case of the monotonicity formula \ref{monoform} (in which $C=0$) to conclude that $u_{\infty}$ must satisfy the $0$-homogeneity condition
$$\langle du_{\infty}(x),x\rangle=0\text{ for a.e. }x\in \mathbb{R}^n.$$
In the next section, we will appeal to this strong convergence to tangent maps to obtain a sharp lower bound for the density $\theta_p(u,x)$ at singular points of $u$, which will form the foundation for many of the estimates that follow.

\section{Sharp $\epsilon$-Regularity and Estimates for $T(u)$}\label{bigests}

\begin{subsection}{A Sharp Lower Bound for Energy Density on $Sing(u)$}\label{sharpereg}\hfill

\hspace{3mm} The analysis leading to Theorem \ref{mainthm1} rests largely on the following proposition--the comparatively simple counterpart in our setting to the ``$\eta$-compactness"/``$\eta$-ellipticity" results of \cite{BBO},\cite{LR}:

\begin{prop}\label{sharpdensbd} Let $u\in W^{1,p}(M^n,S^1)$ be a stationary $p$-harmonic map with $n\geq 2$ and $p\in (1,2)$, and let $x\in Sing(u)$ be a singular point. Then
\begin{equation}
\theta_p(u,x)\geq c(n,p)\frac{2\pi}{2-p},
\end{equation}
where $c(2,p)=1$, and, for $n>2$,
$$c(n,p):=\int_{B_1^{n-2}}(\sqrt{1-|y|^2})^{2-p}dy\to \omega_{n-2}\text{ as }p\to 2.$$
(Here, $\omega_m$ denotes the volume of the Euclidean unit $m$-ball.)
\end{prop}
\begin{proof} Let $x\in Sing(u)$; by the small energy regularity theorem of \cite{TW}, this is equivalent to the positivity of the density $\theta_p(u,x)>0$. Taking a sequence of radii $r_j\to 0$ and considering the blow-up sequence $u_j=u_{x,r_j}$, we know from the discussion in Section \ref{pstatcons} that some subsequence $u_{j_k}$ converges strongly in $W^{1,p}_{loc}(\mathbb{R}^n,S^1)$ to a nontrivial stationary $p$-harmonic map
$$v\in W^{1,p}_{loc}(\mathbb{R}^n,S^1)$$
satisfying
\begin{equation}\label{tanmapconds}
\theta_p(v,0,r)=\theta_p(u,x)\text{ for all }r>0,\text{ and }\frac{\partial v}{\partial r}=0.
\end{equation}
Since the tangent map $v$ is radially homogeneous, it follows that its restriction $v|_S$ to the unit sphere defines a weakly $p$-harmonic map on $S^{n-1}$. 

\hspace{3mm} Next, we observe that if $n>2$, the restriction $v|_S$ must again have a nontrivial singular set. Indeed, if $v|_S$ were $C^1$, then since $H^1_{dR}(S^{n-1})=0$, we would have a lifting $v|_S=e^{i\varphi}$ for some $p$-harmonic function $\varphi\in W^{1,p}(S^{n-1},\mathbb{R})$. The only $p$-harmonic functions on closed manifolds are the constants, so this would contradict the nontriviality of $v$. Thus, $v|_S$ must have nonempty singular set on $S^{n-1}$, and in particular, $Sing(v)$ must contain at least one ray in $\mathbb{R}^n$.

\hspace{3mm} We proceed now by a simple dimension reduction-type argument. Fixing some singular point $x_1\in Sing(v)\setminus \{0\}$ of $v$ away from the origin, a standard application of (\ref{tanmapconds}) and the monotonicity formula gives the density inequality
\begin{equation}
\theta_p(v,x_1)\leq \theta_p(v,0)=\theta_p(u,x).
\end{equation}
Thus, we can take a blow-up sequence for $v$ at $x_1$ to obtain a new tangent map $v_1\in W^{1,p}_{loc}(\mathbb{R}^n,S^1)$ satisfying
\begin{equation}
0<\theta_p(v_1,0,r)=\theta_p(v_1,0)\leq \theta_p(u,x)\text{ for all }r>0.
\end{equation}
This map $v_1$ will again be radially homogeneous (by the monotonicity formula), and from the radial homogeneity $\frac{\partial v}{\partial r}=0$ of $v$, $v_1$ inherits the additional translation symmetry $\langle dv_1,x_1\rangle=0$ in the $x_1$ direction. In particular, $v_1$ is determined by its restriction to an $(n-2)$-sphere in the hyperplane perpendicular to $x_1$, which defines a weakly $p$-harmonic map from $S^{n-2}$ to $S^1$. 

\hspace{3mm} If $n-2>1$, we can argue as before to see that $v_1$ must have singularities on this $(n-2)$-sphere, and blow up again at some point $x_2\in Sing(v_1)\setminus \mathbb{R}x_1$. Carrying on in this way, we obtain finally a nontrivial stationary $p$-harmonic map $v_{n-2}\in W_{loc}^{1,p}(\mathbb{R}^n,S^1)$ which is radially homogeneous, invariant under translation by some $(n-2)$-plane $\mathcal{L}^{n-2}$, and satisfies
\begin{equation}\label{findenslbd}
\theta_p(v_{n-2},0,r)\leq \theta_p(u,x)\text{ for all }r>0.
\end{equation}
Now, it's easy to see that the only weakly $p$-harmonic maps from $S^1$ to $S^1$ are given by the identity $z\mapsto z$ and its powers $z\mapsto z^{\kappa}$ for $\kappa\in \mathbb{Z}$. In particular, letting $z$ denote the projection of $x$ onto $\mathcal{L}^{\perp}$, it follows that
$$v_{n-2}(x)=(z/|z|)^{\kappa}$$
for some $0\neq \kappa \in \mathbb{Z}$. We can therefore compute
\begin{eqnarray*}
\theta_p(v_{n-2},0,1)&=&\int_{B_1^{n-2}}\int_{D^2_{\sqrt{1-|y|^2}}}\frac{|\kappa|^p}{|z|^p}dzdy\\
&=&\frac{2\pi|\kappa|^p}{2-p}\int_{B_1^{n-2}}(\sqrt{1-|y|^2})^{2-p}dy\\
&=&\frac{2\pi|\kappa|^p}{2-p}c(n,p).
\end{eqnarray*}
It then follows from (\ref{findenslbd}) that
\begin{equation}
\theta_p(u,x)\geq \theta_p(v_{n-2},0,1)\geq c(n,p)\frac{2\pi}{2-p},
\end{equation}
as desired.
\end{proof}

\end{subsection}

\begin{subsection}{Consequences of Proposition \ref{sharpereg} and Estimates for $T(u)$}

\hfill

\hspace{3mm} Throughout this section, let $M$ be an $n$-dimensional manifold satisfying the sectional curvature and injectivity radius bounds
\begin{equation}
|sec(M)|\leq k,\text{ }inj(M)\geq 3,
\end{equation}
and let $p\in [3/2,2)$. As a first consequence of Proposition \ref{sharpereg}, we employ a simple Vitali covering argument (compare, e.g., Theorem 3.5 of \cite{NVV}) to obtain $p$-independent estimates for the $(n-p)$-content of the singular set $Sing(u)=spt(T(u))$ of a stationary $p$-harmonic map $u$ to $S^1$.

\begin{lem}\label{singvollem} Let $u\in W^{1,p}(B_3(x),S^1)$ be a stationary $p$-harmonic map on a geodesic ball $B_3(x)\subset M$ of radius $3$, satisfying the $p$-energy bound
\begin{equation}
E_p(u,B_3(x))\leq \frac{\Lambda}{2-p}.
\end{equation}
For $r\leq 1$, the $r$-tubular neighborhood $\mathcal{N}_r(Sing(u)\cap B_2(x))$ about the singular set $Sing(u)$ then satisfies a volume bound of the form
\begin{equation}
Vol(\mathcal{N}_r(Sing(u))\cap B_2(x))\leq C(k,n)\Lambda r^p.
\end{equation}
\end{lem}

\begin{proof} Applying the Vitali covering lemma to the covering 
$$\{B_r(y)\mid y\in Sing(u)\cap B_2(x)\}$$ 
of $\mathcal{N}_r(Sing(u)\cap B_2(x))$, we obtain a finite subcollection $x_1,\ldots,x_m\in Sing(u)\cap B_2(x)$ for which
\begin{equation}\label{vitdis}
B_r(x_i)\cap B_r(x_j)=\varnothing\text{ when }i\neq j,
\end{equation}
and
\begin{equation}\label{vitcov}
\mathcal{N}_r(Sing(u)\cap B_2(x))\subset \bigcup_{i=1}^mB_{5r}(x_j).
\end{equation}

\hspace{3mm} Now, by virtue of Proposition \ref{sharpereg} and Lemma \ref{mono}, we have for each $B_r(x_i)$ the lower energy bound
$$\frac{2\pi c(n,p)}{2-p}\leq C(k,n)\theta_p(u,x_i,r)=C(k,n)r^{p-n}\int_{B_r(x_i)}|du|^p,$$
and from the disjointness (\ref{vitdis}) of $\{B_r(x_i)\}$, it follows that
$$m\frac{2\pi c(n,p)}{2-p}\leq C(k,n)r^{p-n}\int_{\mathcal{N}_r(Sing(u))\cap B_3(x)}|du|^p\leq C(k,n)r^{p-n}\frac{\Lambda}{2-p}.$$
Since $\inf_{p\in [3/2,2)}c(n,p)>0$, this gives us an estimate of the form
$$m\leq C(k,n)\Lambda r^{p-n}.$$
By virtue of (\ref{vitcov}), it then follows that
\begin{eqnarray*}
Vol(\mathcal{N}_r(Sing(u)\cap B_2(x)))&\leq & m\cdot C'(k,n)r^n\\
&\leq & C(k,n) \Lambda r^p,
\end{eqnarray*}
as claimed.
\end{proof}

\hspace{3mm} For analysis purposes, this volume estimate is one of the most important consequences of Proposition \ref{sharpereg}, and we will use it repeatedly throughout the remainder of the paper. The first application is a series of improved estimates for the distributional Jacobian $T(u)$ of $u$. A priori, without further knowledge of the map $u\in W^{1,p}(M,S^1)$, we know only that
$$\|T(u)\|_{W^{-1,p}}\leq \|du\|_{L^p},$$
since
$$\langle T(u), \zeta\rangle=\int_Mju\wedge d\zeta\leq \|ju\|_{L^p}\|\zeta\|_{W^{1,{p'}}}$$
for every $\zeta\in \Omega^{n-2}(M)$ (with $p':=\frac{p}{p-1}$). With Lemma \ref{singvollem} in hand, however, we are able to show that, for stationary $p$-harmonic $u_p\in W^{1,p}(M,S^1)$, if $E_p(u_p)=O(\frac{1}{2-p})$, then $T(u_p)$ is in fact uniformly bounded in various norms as $p\uparrow 2$. The first step in this direction is the following estimate:

\begin{prop}\label{tubds1} Under the assumptions of Lemma \ref{singvollem}, let $T(u)$ denote the distributional Jacobian of $u$. Then for any smooth $(n-2)$-form $\zeta\in \Omega_c^{n-2}(B_2(x))$ supported in $B_2(x)$, we have 
\begin{equation}
\langle T(u),\zeta\rangle\leq C(k,n)\Lambda \|\zeta\|_{L^{\infty}}^{p-1}\|d\zeta\|_{L^{\infty}}^{2-p}.
\end{equation}
\end{prop}

\begin{proof} Fix $\beta\in (0,1)$ and set $K:=\lfloor \frac{1}{2-p}\rfloor \geq \frac{p-1}{2-p}$. Setting $U(r):=\mathcal{N}_r(Sing(u)\cap B_2(x))$, we begin with the simple estimate
\begin{eqnarray*}
E_p(u,B_2(x))&\geq &\int_{U(\beta)\setminus U(2^{-K}\beta)}|du|^p\\
&=&\Sigma_{j=1}^K\int_{U(2^{1-j}\beta)\setminus U(2^{-j}\beta)}|du|^p,
\end{eqnarray*}
from which it follows that
\begin{equation}
\int_{U(2^{1-j}\beta)\setminus U(2^{-j}\beta)}|du|^p\leq \frac{E_p(u)}{K}.
\end{equation}
for some $j\in \{1,\ldots,K\}$. In particular, there is some scale $s=2^{-j}\beta\in [2^{\frac{-1}{2-p}}\beta,\frac{1}{2}\beta]$ for which
\begin{equation}\label{uannest}
\int_{U(2s)\setminus U(s)}|du|^p\leq \frac{E_p(u)}{K}\leq \frac{\Lambda}{2-p}\cdot\frac{2-p}{p-1}=\frac{\Lambda}{p-1}.
\end{equation}

\hspace{3mm} Now, let $\psi(y)=\eta(dist(y,Sing(u)))$, where $\eta$ is given by
$$\eta\equiv 1\text{ on }[0,s],\text{ }\eta(t)=1-\frac{1}{s}(t-s)\text{ for }t\in [s,2s],\text{ and }\eta\equiv 0\text{ on }[2s,\infty),$$
so that $\psi\equiv 1$ on $U(s)\supset spt(T(u))$ and $\psi \equiv 0$ outside $U(2s)$.
For any $\zeta\in \Omega_c^{n-2}(B_2(x))$, we then have
\begin{eqnarray*}
\langle T(u),\zeta\rangle&=&\langle T(u),\psi\zeta\rangle\\
&=&\int ju\wedge d\psi\wedge \zeta+ ju\wedge \psi d\zeta\\
&\leq & \|d\psi\|_{L^{\infty}}\|ju\|_{L^p(U(2s)\setminus U(s))}\|\zeta\|_{L^{p'}(U(2s))}\\
&&+\|ju\|_{L^p(U(2s))}\|d\zeta\|_{L^{p'}(U(2s))},
\end{eqnarray*}
where $p'=\frac{p}{p-1}$. Next, we note that $\|d\psi\|_{L^{\infty}}=\frac{1}{s}$, while 
$$\|ju\|_{L^p(U(2s))}\leq E_p(u,B_2(x))^{1/p}\leq (2-p)^{-1/p}\Lambda^{1/p},$$
and, by (\ref{uannest}),
$$\|ju\|_{L^p(U(2s)\setminus U(s))}\leq \frac{\Lambda^{1/p}}{(p-1)^{1/p}};$$
using all of this in the preceding estimate, we then obtain
\begin{equation}\label{pretest}
\langle T(u),\zeta\rangle\leq \frac{1}{s}\frac{\Lambda^{1/p}}{(p-1)^{1/p}}\|\zeta\|_{L^{p'}(U(2s))}+(2-p)^{-1/p}\Lambda^{1/p}\|d\zeta\|_{L^{p'}(U(2s))}.
\end{equation}

Now, by Lemma \ref{singvollem}, we see that
$$\|\zeta\|_{L^{p'}(U(2s))}\leq \|\zeta\|_{L^{\infty}}Vol(U(2s))^{1-1/p}\leq C(k,n)\Lambda^{1-1/p}s^{p-1}\|\zeta\|_{L^{\infty}}$$
and 
$$\|d\zeta\|_{L^{p'}(U(2s))}\leq C(k,n)\Lambda^{1-1/p}s^{p-1}\|d\zeta\|_{L^{\infty}},$$
which we use in (\ref{pretest}) to obtain
\begin{equation}
\langle T(u),\zeta\rangle \leq \frac{C\Lambda s^{p-2}}{(p-1)^{1/p}}\|\zeta\|_{L^{\infty}}+\frac{C\Lambda s^{p-1}}{(2-p)^{1/p}}\|d\zeta\|_{L^{\infty}}.
\end{equation}
Recalling that $s$ lies in the interval $2^{\frac{-1}{2-p}}\beta\leq s\leq \frac{1}{2}\beta$, it then follows that 
\begin{equation}\label{arbbetaeq}
\langle T(u),\zeta\rangle \leq C(k,n)\Lambda\left(\frac{\beta^{p-2}\|\zeta\|_{L^{\infty}}}{(p-1)^{1/p}}+\frac{\beta^{p-1}\|d\zeta\|_{L^{\infty}}}{(2-p)^{1/p}}\right).
\end{equation}

\hspace{3mm} Finally, we observe that $\beta\in (0,1)$ was arbitrary, so we can choose, for instance
$$\beta(p,\zeta)=\frac{(2-p)^{1+1/p}}{(p-1)^{1+1/p}}\frac{\|\zeta\|_{L^{\infty}}}{\|d\zeta\|_{L^{\infty}}}.$$
If $\beta(p,\zeta)\geq 1,$ then 
$$\|d\zeta\|_{L^{\infty}}\leq \frac{(2-p)^{1+1/p}}{(p-1)^{1+1/p}}\|\zeta\|_{L^{\infty}},$$
and the desired estimate holds trivially. Otherwise, we have $\beta(p,\zeta)\in (0,1)$, so we can plug $\beta=\beta(p,\zeta)$ into (\ref{arbbetaeq}), and using the fact that $(2-p)^{p-2}$ is uniformly bounded for $p\in (1,2)$, we arrive an estimate of the desired form
\begin{equation}
\langle T(u),\zeta\rangle \leq C(k,n)\Lambda \|\zeta\|_{L^{\infty}}^{p-1}\|d\zeta\|_{L^{\infty}}^{2-p}.
\end{equation}

\end{proof}

\hspace{3mm} By rescaling the result of Proposition \ref{tubds1}, we obtain the following statement at arbitrary small scales $0<r\leq 1$:

\begin{cor}\label{tubdsscaled} For $0<r\leq 1$, let $u\in W^{1,p}(B_{3r}(x),S^1)$ be a stationary $p$-harmonic map with 
$$\theta_p(u,x,3r)\leq \frac{\Lambda}{2-p}.$$
Then for every $\zeta\in \Omega_c^{n-2}(B_{2r}(x))$, we have
\begin{equation}\label{scaledtubdest}
\langle T(u),\zeta\rangle \leq C(k,n)\Lambda r^{n-p}\|\zeta\|^{p-1}_{L^{\infty}}\|d\zeta\|^{2-p}_{L^{\infty}}.
\end{equation}
\end{cor}

\hspace{3mm} In particular, by virtue of energy monotonicity (Lemma \ref{mono}), if $u\in W^{1,p}(B_3,S^1)$ satisfies $E_p(u,B_3)\leq \frac{\Lambda}{2-p}$, then (\ref{scaledtubdest}) holds for all $\zeta\in \Omega_c^{n-2}(B_r(y))$, for every ball $B_r(y)\subset B_2(x)$. In the appendix, we establish the following general lemma, which will imply that these estimates, together with the volume bounds of Lemma \ref{singvollem}, yield $p$-independent bounds for $\|T(u)\|_{W^{-1,q}}$ for any $q\in [1,p)$:

\begin{lem}\label{bigapplem} Let $S$ be an $(n-2)$-current in $W^{-1,p}(B_2(x))$ satisfying
\begin{equation}
\langle S,\zeta\rangle \leq A r^{n-p}\|\zeta\|_{L^{\infty}}^{p-1}\|d\zeta\|_{L^{\infty}}^{2-p}\text{\hspace{3mm}}\forall \zeta\in \Omega_c^{n-2}(B_r(y)),
\end{equation}
for every ball $B_r(y)\subset B_2(x)$. Suppose also that the $r$-tubular neighborhoods $\mathcal{N}_r(spt(S))$ about the support of $S$ satisfy
\begin{equation}
Vol(\mathcal{N}_r(B_2(x)\cap spt(S)))\leq Ar^p.
\end{equation}
Then there is a constant $C(n,k,A)$ such that for every $1\leq q<p$, we have
\begin{equation}
\|S\|_{W^{-1,q}(B_1(x))}\leq C(n,k,A)(p-q)^{-1/q}.
\end{equation}
\end{lem}

In particular, combining the results of Lemma \ref{singvollem} with Lemma \ref{mono} and Corollary \ref{tubdsscaled}, we find that

\begin{cor}\label{tuwbds} Let $u\in W^{1,p}(B_3(x),S^1)$ be a stationary $p$-harmonic map with 
$$E_p(u,B_3(x))\leq \frac{\Lambda}{2-p}.$$
Then for every $q\in [1,p)$, we have
\begin{equation}\label{tuwbdsest}
\|T(u)\|_{W^{-1,q}(B_1(x))}\leq C(n,k,\Lambda,q).
\end{equation}
\end{cor}

\begin{proof} By the preceding discussion, we see that
$$\|T(u)\|_{W^{-1,q}(B_1(x))}\leq C(n,k,\Lambda)(p-q)^{-1/q}.$$
To put this in the form (\ref{tuwbdsest}), we simply separate into two cases: if $p\leq 1+\frac{q}{2}$, then $2-p\geq \frac{2-q}{2}$, and (\ref{tuwbdsest}) follows from the trivial estimate
$$\|T(u)\|_{W^{-1,p}}\leq \|du\|_{L^p}\leq \frac{\Lambda^{1/p}}{(2-p)^{1/p}}.$$
On the other hand, if $p> 1+\frac{q}{2}$, then $p-q>\frac{2-q}{2}$, and so
$$\|T(u)\|_{W^{-1,q}(B_1(x))}\leq C(n,k,\Lambda)(p-q)^{-1/q}\leq C(n,k,\Lambda,q),$$
as claimed.
\end{proof}

\begin{remark} Given the existence of estimates--like those of \cite{JS1}--which bound in certain weak norms the Jacobians of arbitrary maps in $W^{1,2}(M,\mathbb{C})$ by the normalized Ginzburg-Landau energies $\frac{E_{\epsilon}}{|\log\epsilon|}$, it is natural to ask whether results along the lines of Proposition \ref{tubds1} or Corollary \ref{tuwbds} can be obtained for arbitrary maps $u\in W^{1,p}(M,S^1).$ If so, such estimates could be of use for the study of variational problems in $W^{1,p}(M,S^1)$, just as the results of \cite{JS1} have been in the Ginzburg-Landau setting (as in, e.g., \cite{BOSp}, \cite{Ste2}).
\end{remark}

\end{subsection}

\section{Estimates For the Hodge Decomposition of $ju$}\label{hodgeests}

\hspace{3mm} Now, let $M^n$ again be an arbitrary compact, oriented Riemannian manifold, and for $p\in [3/2,2)$, let $u\in W^{1,p}(M,S^1)$ be a stationary $p$-harmonic map with
\begin{equation}
E_p(u)\leq \frac{\Lambda}{2-p}.
\end{equation}
In our analysis of the global behavior of $u$, just as in the Ginzburg-Landau setting, the Hodge decomposition
\begin{equation}\label{hodge}
ju=d\varphi+d^*\xi+h
\end{equation}
of $ju$ plays a central role. (For more on Hodge decomposition in the space of $L^p$ differential forms, we refer the reader to \cite{Scott}.) Here, $\varphi\in W^{1,p}(M,\mathbb{R})$ is the function given by
$$\varphi:=\Delta^{-1}(div(ju)),$$
$\xi$ is the $W^{1,p}$ two-form
$$\xi:=*\Delta_H^{-1}T(u),$$
and $h$ is the remaining harmonic one-form, which we can write as
$$h:=\Sigma_{i=1}^k\left(\int_M\langle h_i, ju\rangle\right) h_i$$
with respect to an $L^2$-orthonormal basis $\{h_i\}_{i=1}^k$ for the space $\mathcal{H}^1(M)$ of harmonic one-forms. We remark that, in our notation, $\Delta$ denotes the negative spectrum scalar Laplacian, but $\Delta_H=dd^*+d^*d$ is the usual positive spectrum Hodge Laplacian.

\hspace{3mm} Our first goal in this section is to establish estimates which show that, for each $q\in [1,2)$, and any sequence $u_p$ as in Theorem \ref{mainthm1}, the coexact component $d^*\xi$ remains bounded and the exact component $d\varphi$ vanishes in $L^q$ as $p\to 2$. The second is to show that the same behavior holds in stronger norms away from the singular sets.

\hspace{3mm} For the harmonic form $h$, we need only the trivial $L^{\infty}$ estimate
\begin{equation}
\|h\|_{L^{\infty}(M)}\leq C(M)\|du\|_{L^1(M)}\leq C(M)\frac{\Lambda^{1/p}}{(2-p)^{1/p}}.
\end{equation}
For the exact and co-exact terms $d\varphi$ and $d^*\xi$, we begin by establishing the following global estimates:

\begin{prop}\label{hodgeglob} If $q\in [1,p)$, then
\begin{equation}\label{xiglob}
\|\xi\|_{W^{1,q}}\leq C(M,\Lambda,q)
\end{equation}
and
\begin{equation}\label{phiglob}
\|\varphi\|_{W^{1,q}}\leq C(M,\Lambda,q)(2-p)^{1-1/p}|\log(2-p)|
\end{equation}
for some $C(M,\Lambda,q)<\infty$ independent of $p$.
\end{prop}

\begin{proof} First, note that we can apply Corollary \ref{tuwbds} (after some fixed rescaling) to a finite covering of $M$ by geodesic balls, to obtain the $W^{-1,q}$ estimate
\begin{equation}
\|T(u)\|_{W^{-1,q}(M)}\leq C(M,\Lambda,q)
\end{equation}
for the distributional Jacobian $T(u)$. Since $\xi:=*\Delta_H^{-1}T(u)$ by definition, it follows from the $L^q$ regularity of $\Delta_H$ that
$$\|\xi\|_{W^{1,q}}\leq C(M,q)\|T(u)\|_{W^{-1,q}}\leq C(M,\Lambda,q),$$
as desired.

\hspace{3mm} To estimate $d\varphi$, we begin by observing that since $u$ is weakly $p$-harmonic, the distributional divergence $$div(|ju|^{p-2}ju)$$ vanishes, and $\varphi$ can therefore be recast as
$$\varphi:=\Delta^{-1}(div(ju-|ju|^{p-2}ju)).$$
The $L^q$ regularity of the Laplacian then gives
\begin{equation}
\|\varphi\|_{W^{1,q}}\leq C(q,M)\|(1-|ju|^{p-2})ju\|_{L^q},
\end{equation}
so it is enough produce an $L^q$ estimate of the desired form for $(1-|ju|^{p-2})ju$.

\hspace{3mm} To this end, we write
\begin{eqnarray*}
\|(1-|ju|^{p-2})ju\|_{L^q}^q&=&\int ||du|-|du|^{p-1}|^q\\
&=&\int_{\{|du|\leq 1\}}(|du|^{p-1}-|du|)^q\\
&&+\int_{\{|du|\geq 1\}}(|du|-|du|^{p-1})^q.
\end{eqnarray*}
It's easy to check that
$$\max_{t\in [0,1]}(t^{p-1}-t)=(2-p)(p-1)^{\frac{p-1}{2-p}},$$
so the $\{|du|\leq 1\}$ portion of the integral satisfies 
$$\int_{\{|du|\leq 1\}}(|du|^{p-1}-|du|)^q\leq (2-p)^q(p-1)^{\frac{q(p-1)}{2-p}}Vol(M).$$

\hspace{3mm} To estimate the $\{|du|\geq 1\}$ portion of the integral, observe that  
$$1-t^{p-2}\leq (2-p)\log(t)$$
when $t\geq 1$, so fixing some $\lambda>1$, we split the integral again to see that
\begin{eqnarray*}
\int_{\{|du|\geq 1\}}(|du|-|du|^{p-1})^q&\leq &\int_{\{1\leq |du|\leq \lambda\}}(2-p)^q\log(|du|)^q|du|^q\\
&&+\int_{\{|du|\geq \lambda\}}|du|^q\\
&\leq &(2-p)^q\log(\lambda)^q\int_{\{1\leq |du|\leq \lambda\}}|du|^q\\
&&+\|du\|_{L^p}^{q/p}Vol(\{|du|\geq \lambda\})^{1-q/p}\\
&\leq &(2-p)^q\log(\lambda)^qC(M)\|du\|_{L^p}^q+\|du\|_{L^p}^p\lambda^{q-p},
\end{eqnarray*}
and since $\|du\|_{L^p}\leq \Lambda^{1/p}(2-p)^{-1/p}$, this yields
$$\int_{\{|du|\geq 1\}}(|du|-|du|^{p-1})^q\leq C(M)\Lambda^{q/p}(2-p)^{q-q/p}\log(\lambda)^q+\Lambda (2-p)^{-1}\lambda^{q-p}.$$
Taking $\lambda=(2-p)^{-\frac{1}{p}-\frac{q}{p-q}},$ we observe that
$$(2-p)^{-1}\lambda^{q-p}=(2-p)^{q-q/p}$$
and 
$$\log(\lambda)=(\frac{1}{p}+\frac{q}{p-q})|\log(2-p)|,$$
so putting this together with the preceding inequalities, we arrive at the estimate
\begin{equation}\label{fofdulq}
\int ||du|-|du|^{p-1}|^q\leq \frac{C(M,\Lambda)}{(p-q)^q}(2-p)^{q-q/p}|\log(2-p)|^q.
\end{equation}
Considering separately the cases $p>1+\frac{q}{2}$ and $p\leq 1+\frac{q}{2}$ as in the proof of Corollary \ref{tuwbds}, and recalling that 
$$\|d\varphi\|_{L^q}\leq C(M,q)\||du|-|du|^{p-1}\|_{L^q},$$
we arrive at an estimate of the desired form (\ref{phiglob}).

\end{proof}

\hspace{3mm} Next, we establish estimates resembling (\ref{xiglob}) and (\ref{phiglob}) in $W^{1,2}$ norms away from the singular set $Sing(u)$. The simple estimates of Lemma \ref{offsingest} below by no means represent the optimal bounds of this kind, but they will suffice for the purposes of this paper.

\begin{lem}\label{offsingest} Suppose now that $p\in [\max\{q_n,3/2\},2)$, where $q_n=\frac{2n}{n+2}$ (so that $W^{1,q_n}\hookrightarrow L^2$ by Sobolev embedding). Letting $r(x):=dist(x,Sing(u))$, we have the $L^2$ estimates
\begin{equation}\label{coexl2}
\|r(x)d^*\xi\|_{L^2(M)}\leq C(M,\Lambda)
\end{equation}
and
\begin{equation}\label{exactl2}
\|r(x)d\varphi\|_{L^2(M)}\leq C(M,\Lambda)(2-p)^{1-1/p}|\log(2-p)|
\end{equation}
\end{lem}

\begin{proof} For $\delta>0$, let $\psi_{\delta}(x)\in Lip(M)$ be given by
$$\psi_{\delta}(x)=\max\{0,r(x)-\delta\},$$
so that $\psi_{\delta}\equiv 0$ on a neighborhood of $Sing(u)=spt(T(u))$, and $Lip(\psi_{\delta})\leq 1.$
Then $d^*\xi$ is closed on the support of $\psi_{\delta}$, and it follows that
\begin{eqnarray*}
\int \psi_{\delta}^2|d^*\xi|^2&=&\int \langle d^*\xi,\psi_{\delta}^2d^*\xi\rangle\\
&=&\int \langle \xi,2\psi_{\delta}d\psi_{\delta}\wedge d^*\xi\rangle\\
&\leq &2\int |\psi_{\delta}d^*\xi||\xi|.
\end{eqnarray*}
Now, since $p>q_n$, we have by Sobolev embedding and Proposition \ref{hodgeglob} the estimate
$$\|\xi\|_{L^2}\leq C(M)\|\xi\|_{W^{1,q_n}}\leq C(M,\Lambda);$$
applying this in the preceding inequality, it follows that
$$\int \psi_{\delta}^2|d^*\xi|^2\leq C(M,\Lambda)^2.$$
Taking $\delta\to 0$ and appealing to the monotone convergence theorem, we arrive at (\ref{coexl2}).

\hspace{3mm} For (\ref{exactl2}), we proceed similarly: with $\psi_{\delta}$ defined as above, we use the equation
$$\Delta\varphi=div((1-|ju|^{p-2})ju)$$
(and the fact that $Lip(\psi_{\delta})\leq 1$) to estimate
\begin{eqnarray*}
\int \psi_{\delta}^2|d\varphi|^2&=&\int \langle d\varphi, d(\psi_{\delta}^2\varphi)\rangle-2\langle \psi_{\delta}d\varphi,\varphi d\psi_{\delta}\rangle\\
&=&\int \langle (1-|ju|^{p-2})ju, d(\psi_{\delta}^2\varphi)\rangle\\
&&-2\int \langle \psi_{\delta}d\varphi,\varphi d\psi_{\delta}\rangle\\
&\leq & \|\psi_{\delta}(1-|ju|^{p-2})ju\|_{L^2}(\|\psi_{\delta}d\varphi\|_{L^2}+\|\varphi\|_{L^2})\\
&&+2\|\psi_{\delta}d\varphi\|_{L^2}\|\varphi\|_{L^2}.
\end{eqnarray*}
With a few applications of Young's inequality, it then follows that
\begin{equation}\label{psidphi1}
\|\psi_{\delta}d\varphi\|_{L^2}^2\leq 10(\|\psi_{\delta}(1-|du|^{p-2})du\|_{L^2}^2+\|\varphi\|_{L^2}^2)
\end{equation}
Now, by Sobolev embedding and Proposition \ref{hodgeglob}, we know that
\begin{equation}\label{obvphil2}
\|\varphi\|_{L^2}^2\leq C(M)\|\varphi\|_{W^{1,q_n}}^2\leq C(M,\Lambda)(2-p)^{2-2/p}|\log(2-p)|^2,
\end{equation}
so all that remains is to estimate $\|\psi_{\delta}(1-|du|^{p-2})du\|_{L^2}$.

\hspace{3mm} To this end, observe that the gradient estimate of Corollary \ref{wkpharmreg} implies
$$r(x)^p|du|^p(x)\leq \theta_p(u,x,r(x)),$$
which together with the monotonicity of $\theta_p(u,x,\cdot)$ yields the pointwise gradient estimate
$$r(x)^p|du|^p(x)\leq C(M)\frac{\Lambda}{2-p};$$
in particular, it follows that 
\begin{equation}\label{psigradest}
\psi_{\delta}^2|du|^2\leq C(M,\Lambda)(2-p)^{-2/p}.
\end{equation}
As in the proof of Proposition \ref{hodgeglob}, we note that $|(1-|du|^{p-2})du|\leq (2-p)$ when $|du|\leq 1$, so that
$$\int_{\{|du|\leq 1\}}\psi_{\delta}^2|(1-|du|^{p-2})du|^2\leq C(M)(2-p)^2.$$
Where $|du|\geq 1$, we can make repeated use of the pointwise estimate (\ref{psigradest}), together with the fact that 
$$1-|du|^{p-2}\leq (2-p)\log(|du|)$$
to find 
\begin{eqnarray*}
\int_{\{|du|\geq 1\}} \psi_{\delta}^2(1-|du|^{p-2})^2|du|^2&\leq &C(M,\Lambda)(2-p)^{-2/p}\int_{\{|du|\geq 1\}} (1-|du|^{p-2})^2\\
&\leq &C(M,\Lambda)(2-p)^{2-2/p}\int_{\{|du|\geq 1\}}\log(|du|)^2\\
&\leq & C(M,\Lambda)(2-p)^{2-2/p}\int \log\left(\frac{C(M,\Lambda)}{(2-p)^{1/p}r(x)}\right)^2.
\end{eqnarray*}
Splitting up the logarithm
$$\log\left(\frac{C(M,\Lambda)}{(2-p)^{1/p}r(x)}\right)=\log(C(M,\Lambda))+\frac{1}{p}|\log(2-p)|-\log(r(x)),$$
we then see that
\begin{eqnarray*}
\int_M \psi_{\delta}^2(1-|du|^{p-2})^2|du|^2&\leq & C(M,\Lambda)(2-p)^{2-2/p}|\log(2-p)|^2\\
&&+C(M,\Lambda)(2-p)^{2-2/p}\int_M \log(r(x))^2.
\end{eqnarray*}
Finally, by the volume estimates of Lemma \ref{singvollem}, we know that
$$\|r(x)^{-1}\|_{L^{p,\infty}}\leq C(M,\Lambda),$$
so that
$$\int_M \log(r(x))^2\leq C(M)\int_M r(x)^{-1}\leq C(M,\Lambda).$$
In particular, it then follows that
$$\|\psi_{\delta}(1-|du|^{p-2})du\|_{L^2}^2\leq C(M,\Lambda)(2-p)^{2-2/p}|\log(2-p)|^2,$$
which together with (\ref{psidphi1}) and (\ref{obvphil2}) gives
\begin{equation}
\|\psi_{\delta}d\varphi\|_{L^2}^2\leq C(M,\Lambda)(2-p)^{2-2/p}|\log(2-p)|^2.
\end{equation}
As before, we now take $\delta\to 0$ and appeal to the Monotone Convergence Theorem to arrive at the desired estimate (\ref{exactl2}).
\end{proof}

\section{Limiting Behavior of the $p$-Energy Measure}\label{sec5}

\subsection{Generalized Varifolds} \hfill

\hspace{3mm} Let $M^n$ be a compact Riemannian manifold. For $m\in \mathbb{N}$, denote by $A_m(M^n)$ the compact subbundle
\begin{equation}
A_m(M):=\{S\in End(TM)\mid S=S^*,\text{ }-n I\leq S\leq I,\text{ }tr(S)=m\}
\end{equation}
of $End(TM)$ consisting of symmetric endomorphisms with trace $m$ and eigenvalues lying in $[-n,1]$. In \cite{AS}, Ambrosio and Soner define the space $\mathcal{V}_m'(M)$ of \emph{generalized $m$-varifolds} to be the space of nonnegative Radon measures on $A_m(M)$. Note that $\mathcal{V}_m'(M)$ contains the standard $m$-varifolds--Radon measures on the Grassmannian bundle $G_m(M)$--since identifying subspaces with the associated orthogonal projections gives a natural inclusion $G_m(M)\hookrightarrow A_m(M)$.

\hspace{3mm} As with standard varifolds (see \cite{All},\cite{Simon} for an introduction), for any $V\in \mathcal{V}_m'$, we define the weight measure $\|V\|$ to be the pushforward $\pi_*V$ of $V$ under the projection $\pi: A_m(M)\to M$, and the first variation $\delta V$ to be the functional on $C^1$ vector fields given by
\begin{equation}\label{bddvar}
\delta V(X):=\int_{A_m(M)}\langle S, \nabla X\rangle dV(S).
\end{equation}
A classical result of Allard (see \cite{All}, Section 5) states that any (standard) $m$-varifold $V$ whose first variation $\delta V$ is bounded in the $(C^0)^*$ sense
$$\delta V(X)\leq C_V\|X\|_{C^0}$$
restricts to an $m$-rectifiable varifold on the set $\{x\mid \Theta^*_m(\|V\|,x)>0\}$ where its (upper-)$m$-dimensional density
$$\Theta_m^*(\|V\|,x):=\liminf_{r\to 0}\frac{\|V\|(B_r(x))}{\omega_mr^m}$$
is positive. In \cite{AS}, this result is extended to the setting of generalized varifolds as follows:

\begin{prop}\label{asprop}\emph{(\cite{AS})} Let $V\in \mathcal{V}_m'(M)$ be a generalized $m$-varifold with bounded first variation
\begin{equation}
\delta V(X)\leq C_V\|X\|_{C^0}
\end{equation}
and positive $m$-density
\begin{equation}
\Theta_m^*(\|V\|,x)>0\text{ for }\|V\|-a.e.\text{ }x\in M.
\end{equation}
Then there is an $m$-rectifiable (classical) varifold $\tilde{V}$ such that
\begin{equation}
\|\tilde{V}\|=\|V\|\text{ and }\delta \tilde{V}=\delta V.
\end{equation}
\end{prop}

In \cite{AS}, this result was originally used to study concentration of energy for solutions of the parabolic Ginzburg-Landau equations, as a means for constructing codimension-two Brakke flows. In the proof of Theorem \ref{mainthm1}, we will use it similarly, to show that the concentrated part of $\mu$ is given by the weight measure of a stationary, rectifiable $(n-2)$-varifold.

\subsection{Proof of Theorems \ref{mainthm1} and \ref{mapconvthm}}
\hfill

\hspace{3mm} As in Theorem \ref{mainthm1}, let $M$ be a compact, oriented Riemannian manifold, let $p_i\in (1,2)$ with $\lim_{i\to\infty}p_i=2$, and let $u_i\in W^{1,p_i}(M,S^1)$ be a sequence of stationary $p_i$-harmonic maps satisfying
\begin{equation}
\Lambda:=\sup_{i\in \mathbb{N}}\int_M(2-p_i)|du_i|^{p_i}dv_g<\infty.
\end{equation}
Passing to a subsequence, we can assume also that the $p_i$-energy measures
$$\mu_i:=(2-p_i)|du_i|^{p_i}dv_g$$
converge in $(C^0)^*$ to a limiting measure $\mu$, and that the singular sets $Sing(u_i)$ converge in the Hausdorff metric to a limiting set
$$\Sigma=\lim_{i\to\infty}Sing(u_i).$$

\hspace{3mm} Now, for each $i$, consider as in Section \ref{hodgeests} the Hodge decomposition
$$ju_i=d^*\xi_i+d\varphi_i+h_i$$
of $ju_i$, and set $\alpha_i:=d^*\xi_i+d\varphi_i$. We associate to $ju_i$, $\alpha_i,$ and $h_i$, the following $L^1$ sections of $End(TM)$:
$$S_i:=|du_i|^{p_i-2}du_i^*du_i=|ju_i|^{p_i-2}ju_i\otimes ju_i,$$
$$S_i^s:=|\alpha_i|^{p_i-2}\alpha_i\otimes \alpha_i,$$
and
$$S_i^h:=|h_i|^{p_i-2}h_i\otimes h_i.$$
As we shall see, the proof of Theorem \ref{mainthm1} rests largely on the following simple claim:
\begin{claim} 
\begin{equation}\label{claimstat}
\lim_{i\to\infty}(2-p_i)\|S_i-S_i^s-S_i^h\|_{L^1}=0.
\end{equation}
\end{claim}
\begin{proof}
Denoting by $f: T^*_xM\to End(T_xM)$ the function
$$f(X)=|X|^{p-2}X\otimes X$$
for $p\in (1,2)$, it is easy to check that 
$$|\nabla f(X)|\leq 3|X|^{p-1};$$ 
as an immediate consequence, we then have
\begin{eqnarray*}
|f(X+Z)-f(X)|&\leq &\int_0^13|X+tZ|^{p-1}|Z|dt\\
&\leq & 6(|X|^{p-1}+|Z|^{p-1})|Z|
\end{eqnarray*}
for any $X,Z\in T_pM$. In particular, since $S_i=f(ju_i)=f(\alpha_i+h_i)$, $S_i^s=f(\alpha_i)$, and $S_i^h=f(h_i)$ (with $p=p_i$) by definition, it follows that
$$|S_i-S_i^s|\leq 6(|du_i|^{p_i-1}+|h_i|^{p_i-1})|h_i|$$
and
$$|S_i-S_i^h|\leq 6(|du_i|^{p_i-1}+|\alpha_i|^{p_i-1}))|\alpha_i|.$$

\hspace{3mm} With this in mind, we estimate the $L^1$ norm of $S_i-S_i^s-S_i^h$ by splitting $M$ into $\mathcal{N}_{\delta}(Sing(u_i))$ and $M\setminus\mathcal{N}_{\delta}(Sing(u_i))$ for $\delta>0$ small, writing
\begin{eqnarray*}
\|S_i-S_i^s-S_i^h\|_{L^1(M)}&\leq & \int_{\mathcal{N}_{\delta}(Sing(u_i))}|S_i-S_i^s|+|S_i^h|\\
&&+\int_{M\setminus\mathcal{N}_{\delta}(Sing(u_i))}|S_i-S_i^h|+|S_i^s|\\
&\leq & \int_{\mathcal{N}_{\delta}(Sing(u_i))}(6|du_i|^{p_i-1}|h_i|+7|h_i|^{p_i})\\
&&+\int_{M\setminus \mathcal{N}_{\delta}(Sing(u_i))}(6|du_i|^{p_i-1}|\alpha_i|+7|\alpha_i|^{p_i-1}).
\end{eqnarray*}
Now, since
$$\|h_i\|_{L^{\infty}}^{p_i}\leq C(M)\frac{\Lambda}{2-p_i}$$
and, by Lemma \ref{singvollem},
$$Vol(\mathcal{N}_{\delta}(Sing(u_i)))\leq C(M,\Lambda)\delta^{p_i},$$
we have the simple estimate
\begin{equation}
\|h_i\|_{L^{p_i}(\mathcal{N}_{\delta}(Sing(u_i)))}^{p_i}\leq \frac{C(M,\Lambda)}{2-p_i}\delta^{p_i}.
\end{equation}
On the other hand, we know from Lemma \ref{offsingest} that
\begin{eqnarray*}
\int_{M\setminus\mathcal{N}_{\delta}(Sing(u_i))}|\alpha_i|^{p_i}&\leq & C(M)\left(\delta^{-2}\int_{M\setminus\mathcal{N}_{\delta}(Sing(u_i))}dist(x,Sing(u_i))^2|\alpha_i|^2\right)^{p_i/2}\\
&\leq & C(M,\Lambda)\delta^{-p_i}.
\end{eqnarray*}
Returning to our estimate for $\|S_i-S_i^s-S_i^h\|_{L^1(M)}$, it then follows that
\begin{eqnarray*}
\|S_i-S_i^s-S_i^h\|_{L^1(M)}&\leq &\int_{\mathcal{N}_{\delta}(Sing(u_i))}(6|du_i|^{p_i-1}|h_i|+7|h_i|^{p_i})\\
&&+\int_{M\setminus \mathcal{N}_{\delta}(Sing(u_i))}(6|du_i|^{p_i-1}|\alpha_i|+7|\alpha_i|^{p_i-1})\\
&\leq & C\|du_i\|_{L^{p_i}}^{p_i-1}\left(\|h_i\|_{L^{p_i}(\mathcal{N}_{\delta}(Sing(u_i)))}+\|\alpha_i\|_{L^{p_i}(M\setminus \mathcal{N}_{\delta}(Sing(u_i)))}\right)\\
&\leq & C(M,\Lambda)(2-p_i)^{1/p_i-1}(\delta (2-p_i)^{-1/p_i}+\delta^{-1}).
\end{eqnarray*}
Multiplying by $(2-p_i)$ and taking $i\to\infty$, we arrive at the bound
$$\limsup_{i\to\infty}(2-p_i)\|S_i-S_i^s-S_i^h\|_{L^1(M)}\leq C(M,\Lambda)\delta;$$
since $\delta>0$ was arbitrary, (\ref{claimstat}) follows.
\end{proof}

\hspace{3mm} With this claim established, we next observe that measures $\mu_i$ can be written as
$$\mu_i=(2-p_i)tr(S_i)dv_g,$$
and as a consequence of (\ref{claimstat}), we see that
\begin{eqnarray*}
\mu&=&\lim_{i\to\infty} (2-p_i)tr(S_i^s+S_i^h)dv_g\\
&=&\lim_{i\to\infty} [(2-p_i)|\alpha_i|^{p_i}dv_g+(2-p_i)|h_i|^{p_i}dv_g]\\
&=&\lim_{i\to\infty}[(2-p_i)|\alpha_i|^{p_i}dv_g+|\bar{h}_i|^{p_i}dv_g,
\end{eqnarray*}
where in the last line we've set
$$\bar{h}_i:=(2-p_i)^{-1/p_i}h_i.$$
Now, since $\{\bar{h}_i\}$ forms a bounded sequence in the space $\mathcal{H}^1(M)$ of harmonic one-forms, by passing to a further subsequence, we can assume that it converges to some limit
$$\bar{h}=\lim_{i\to\infty}\bar{h}_i\in \mathcal{H}^1(M).$$
It's then clear that 
$$|\bar{h}_i|^{p_i}\to |\bar{h}|^2$$
pointwise, and we can therefore write
\begin{equation}
\mu=\lim_{i\to\infty}(2-p_i)|\alpha_i|^{p_i}dv_g+|\bar{h}|^2dv_g.
\end{equation}

\hspace{3mm} To complete the proof of Theorem \ref{mainthm1}, it remains to realize the measure
$$\nu:=\lim_{i\to\infty}|\alpha_i|^{p_i}dv_g$$
as the weight measure of a stationary, rectifiable $(n-2)$-varifold satisfying the stated properties. To this end, we begin by remarking that, where $|\alpha_i|>0$, the tensor
$$I-2|\alpha_i|^{-2}\alpha_i\otimes \alpha_i\in End(TM)$$
belongs to $A_{n-2}(M)$, so we can define a sequence of generalized $(n-2)$-varifolds $V_i\in \mathcal{V}_{n-2}'(M)$ by 
\begin{equation}
\langle V_i,f\rangle:=\int_M (2-p_i)|\alpha_i|^{p_i}f(I-2|\alpha_i|^{-2}\alpha_i\otimes \alpha_i)\text{ for }f\in C^0(A_{n-2}(M)).
\end{equation}
The associated weight measures $\|V_i\|$ are then given by
$$\|V_i\|:=(2-p_i)|\alpha_i|^{p_i}dv_g,$$
and since we have a uniform mass bound
$$\sup_i\|V_i\|(M)<\infty,$$
we can pass to a further subsequence to obtain a weak limit
$$V=\lim_{i\to\infty}V_i \in \mathcal{V}_{n-2}'(M)$$
with weight measure
$$\|V\|=\nu.$$

\hspace{3mm} We claim next that $\delta V=0$. To see this, let $X$ be a $C^1$ vector field, so that
\begin{eqnarray*}
\delta V(X)&=&\lim_{i\to\infty}\delta V_i(X)\\
&=&\lim_{i\to\infty}\int_M(2-p_i)|\alpha_i|^{p_i}\langle I-2|\alpha_i|^{-2}\alpha_i\otimes \alpha_i,\nabla X\rangle\\
&=&\lim_{i\to\infty}\int_M (2-p_i)|\alpha_i|^{p_i}div(X)-(2-p_i)\langle 2S_i^s,\nabla X\rangle.
\end{eqnarray*}
Appealing once more to (\ref{claimstat}), we then see that
\begin{eqnarray*}
\delta V(X)&=&\lim_{i\to\infty}\int_M(2-p_i)(|du_i|^{p_i}-|h_i|^{p_i})div(X)-(2-p_i)2\langle S_i-S_i^h,\nabla X\rangle\\
&=&\lim_{i\to\infty}(2-p_i)\int_M|du_i|^{p_i}div(X)-2\langle|du_i|^{p_i-2}du_i^*du_i,\nabla X\rangle\\
&&+\lim_{i\to\infty}\int_M|\bar{h}_i|^{p_i}div(X)-2\langle |\bar{h}_i|^{p_i-2}\bar{h_i}\otimes \bar{h}_i,\nabla X\rangle.
\end{eqnarray*}
Now, it's clear that
\begin{eqnarray*}
\lim_{i\to\infty}\int_M|\bar{h}_i|^{p_i}div(X)-2\langle |\bar{h}_i|^{p_i-2}\bar{h}_i\otimes \bar{h}_i,\nabla X\rangle&=&\int_M|\bar{h}|^2div(X)-2\langle \bar{h}\otimes \bar{h},\nabla X\rangle\\
&=& 0,
\end{eqnarray*}
since $div(|\bar{h}|^2I-2\bar{h}\otimes \bar{h})=0$ for harmonic $\bar{h}$. On the other hand, we know from the $p_i$-stationarity of $u_i$ that 
$$\int_M|du_i|^{p_i}div(X)-p\langle |du_i|^{p_i-2}du_i^*du_i,\nabla X\rangle=0,$$
and consequently
\begin{eqnarray*}
|\delta V(X)|&=&|\lim_{i\to\infty}(2-p_i)\int_M(p_i-2)\langle |du_i|^{p_i-2}du_i\otimes du_i,\nabla X\rangle|\\
&\leq &\lim_{i\to\infty} (2-p_i)\Lambda |\nabla X|_{C^0}\\
&=&0,
\end{eqnarray*}
as claimed. 

\hspace{3mm} Since $\nu=\|V\|$ for a generalized $(n-2)$-varifold $V$ with $\delta V=0$, it will follow from Proposition \ref{asprop} that $\nu$ is indeed the weight measure of a stationary, rectifiable $(n-2)$-varifold, once we show that $\nu$ satisfies 
$$\Theta_{n-2}^*(\nu,x)>0\text{ for all }x\in spt(\nu).$$
In particular, to complete the proof of Theorem \ref{mainthm1}, it now suffices to establish the following lemma:

\begin{lem} The support $spt(\nu)$ of $\nu$ is given by
\begin{equation}\label{sptnuchar}
spt(\nu)=\Sigma=\lim_{i\to\infty}Sing(u_i),
\end{equation}
and for $x\in \Sigma$, the density of $\nu$ satisfies the lower bound
\begin{equation}\label{lwrthetabd}
\Theta_{n-2}^*(\nu,x)\geq 2\pi.
\end{equation}
\end{lem}

\begin{proof} To establish (\ref{sptnuchar}), first consider $x\in M\setminus \Sigma$, and set $\delta=dist(x,\Sigma)$. By definition of Hausdorff convergence, it follows that
$$dist(x,Sing(u_i))>\frac{\delta}{2},$$
and consequently 
$$B_{\delta/4}(x)\subset M\setminus \mathcal{N}_{\delta/4}(Sing(u_i)),$$
for $i$ sufficiently large. Appealing once more to the estimates of Lemma \ref{offsingest}, we then see that
\begin{eqnarray*}
\nu(B_{\delta/4}(x))&\leq &\liminf_{i\to\infty}(2-p_i)\int_{B_{\delta/4}(x)}|\alpha_i|^{p_i}\\
&\leq &\lim_{i\to\infty}(2-p_i)\frac{C(M,\Lambda)}{\delta^2}\\
&=&0,
\end{eqnarray*}
so that $x\notin spt(\nu)$; and since $x\in M\setminus \Sigma$ was arbitrary, we therefore have
$$spt(\nu)\subset \Sigma.$$

\hspace{3mm} Next, for $x\in \Sigma$, we'll show that
\begin{equation}\label{muthetalbd}
\Theta_{n-2}^*(\mu,x)\geq 2\pi.
\end{equation}
Indeed, if $x\in \Sigma$, then by definition there is a sequence $x_i\in Sing(u_i)$ for which
$$x=\lim_{i\to\infty}x_i.$$
By Lemma \ref{sharpdensbd}, at each $x_i$, we have 
\begin{equation}
\lim_{r\to 0}(2-p_i)\int_{B_r(x_i)}|du_i|^{p_i}\geq 2\pi c(n,p_i),
\end{equation}
where $c(n,p_i)\to \omega_{n-2}$ as $p_i\to 2$. In particular, fixing $\delta>0$ and appealing to the monotonicity of the $p$-energy (Lemma \ref{mono}), we conclude that
\begin{eqnarray*}
\mu(B_{\delta}(x))&\geq &\liminf_{i\to\infty}\mu_i(B_{\delta-\delta^2}(x_i))\\
&\geq & \lim_{i\to\infty}e^{-C(M)\delta^2} 2\pi c(n,p_i)(\delta-\delta^2)^{n-p_i}\\
&=&e^{-C(M)\delta^2}2\pi\omega_{n-2}(\delta-\delta^2)^{n-2}.
\end{eqnarray*}
Dividing through by $\omega_{n-2}\delta^{n-2}$ and letting $\delta\to 0$, we arrive at the desired lower bound (\ref{muthetalbd}). 

\hspace{3mm} Finally, since the difference 
$$\mu-\nu=|\bar{h}|^2dv_g$$
clearly satisfies 
$$\lim_{\delta\to 0}\delta^{2-n}\int_{B_{\delta}(x)}|\bar{h}|^2dv_g=0,$$
we see that (\ref{muthetalbd}) yields directly the desired density bound (\ref{lwrthetabd}) for $\nu$ on $\Sigma$. Moreover, it follows immediately from (\ref{lwrthetabd}) that $\Sigma\subset spt(\nu)$, and since we've already shown that $spt(\nu)\subset \Sigma$, this completes the proof of (\ref{sptnuchar}) as well. 
\end{proof}

\hspace{3mm} With the proof of Theorem \ref{mainthm1} completed, we turn our attention now to the proof of Theorem \ref{mapconvthm}, concerning compactness of the maps. Suppose that our sequence $u_i\in W^{1,p_i}(M,S^1)$ either satisfies the additional bound
$$\sup_i\|du_i\|_{L^1(M)}\leq C,$$
or that the first Betti number $b_1(M)=0$. In either case, it follows that the harmonic component $h_i$ of $ju_i$ is uniformly bounded
\begin{equation}\label{hibound}
\sup_i\|h_i\|_{L^{\infty}}\leq C
\end{equation}
as $i\to\infty$. Together with the $L^q$ estimates of Proposition \ref{hodgeglob}, this implies immediately that
$$\limsup_{i\to\infty}\|du_i\|_{L^q}<\infty$$
for any $q\in [1,2)$, so some subsequence of $\{u_i\}$ must converge weakly in $W^{1,q}(M,S^1)$ to some limiting map $v$. Moreover, since (by Proposition \ref{hodgeglob}) the exact component $d\varphi_i$ of $ju_i$ vanishes in $L^q$ as $i\to\infty$, it follows that 
$$d^*ju_i\to 0$$
weakly as $i\to\infty$, so the map $v$ must satisfy $d^*jv=0$ distributionally.

\hspace{3mm} Moreover, combining (\ref{hibound}) with Lemma \ref{offsingest}, it follows that, away from any $0<\delta$-neighborhood $\mathcal{N}_{\delta}(\Sigma)$ of $\Sigma$, we have
$$\limsup_{i\to\infty}\int_{M\setminus \mathcal{N}_{\delta}(\Sigma)}|du_i|^2<\infty;$$
and putting this together with the local $W^{2,p}$ estimate of Corollary \ref{wkpharmreg}, we see that
$$\limsup_{i\to\infty}\|u_i\|_{W^{2,p_i}(M\setminus \mathcal{N}_{\delta}(\Sigma))}<\infty.$$
Of course, for $p>\frac{2n}{n+2}$, Rellich's theorem gives us compactness of the embedding $W^{2,p}\hookrightarrow W^{1,2}$; hence, since $p_i>\frac{2n}{n+2}$ for $i$ sufficiently large, there is indeed some subsequence of $\{u_i\}$ which converges strongly in $W^{1,2}(M\setminus \mathcal{N}_{\delta}(\Sigma),S^1)$ to the limiting map $v$ identified above. And since $d^*jv=0$ and $v\in W^{1,2}_{loc}(M\setminus \Sigma,S^1)$, it follows that $v$ is indeed a strongly harmonic map in $C^{\infty}_{loc}(M\setminus \Sigma,S^1)$.

\section{Integrality of the Concentration Measure in Dimension 2}\label{sec6}

\hspace{3mm} By a simple blow-up argument, to establish the quantization result of Theorem \ref{2dintthm}, it is enough to show the following: Let $g_i$ be a sequence of metrics converging (in $C^{\infty}$, say) to the Euclidean one on the disk $D_2(0)\subset\mathbb{R}^2$ of radius $2$, and let $u_i\in W^{1,p_i}(M,S^1)$ be a sequence of stationary $p_i$-harmonic maps (with respect to $g_i$) for which $E_{p_i}(u_i,D_2)\leq \frac{\Lambda}{2-p_i}$ as $p_i\uparrow 2$. If the normalized energy measures $\mu_i=(2-p_i)|du_i|^{p_i}dvol_{g_i}$ converge to a multiple $\theta\delta_0$ of the Dirac mass at $0$, then this multiple $\theta=2\pi k$ for some $k\in \mathbb{N}$. For simplicity, we restrict ourselves here to the case in which each $g_i$ is the flat metric, so that the main result of the section reads as follows:

\begin{thm}\label{sec6thm} Let $u_i\in W^{1,p_i}(D_2(0),S^1)$ be a sequence of stationary $p_i$-harmonic maps from $D_2(0)$ for which
$$\mu_i=(2-p_i)|du_i|^{p_i}(z)dz\to \theta\delta_0$$
weakly in $(C^0(D_2(0)))^*$ as $i\to\infty$. Then $\theta\in 2\pi\mathbb{N}$.
\end{thm}

\hspace{3mm} Before diving in to the proof, we first note that one can easily follow the arguments of the preceding sections to establish local versions of Theorems \ref{mainthm1} and \ref{mapconvthm}. (The estimates of Section \ref{bigests} are already stated in local form, and the estimates of Section \ref{hodgeests} are easily adapted to suitable local variants of the Hodge decomposition, like that which we employ later in this section in the proof of Theorem \ref{sec6thm}.) In particular, for families of stationary $p$-harmonic maps $u_p\in W^{1,p}(D_2(0),S^1)$ on the disk $D_2(0)$, we have the following:

\begin{prop}\label{locdim2} Let $u_i\in W^{1,p_i}(D_2(0),S^1)$ be a sequence of stationary $p_i$-harmonic maps from $D_2(0)$ to $S^1$, for which
$$\sup_i \int_{D_2(0)}(2-p_i)|du_i|^{p_i}<\infty,$$
and 
$$Sing(u_i)\subset D_1(0).$$
A subsequence of the measures $\mu_i=(2-p_i)|du_i|^{p_i}dv_g$ then converges weakly in $(C^0_c(D_2))^*$ to a measure $\mu$ of the form
$$\mu=|d\psi|^2dv_g+\Sigma_{a\in \Sigma}\theta_a\delta_a,$$
where $\psi$ is a harmonic function and $\Sigma$ is given by the Hausdorff limit of $Sing(u_i)$ in $D_1(0)$. Moreover, if 
$$\sup_i\|du_i\|_{L^1(D_2(0))}<\infty,$$
then (a subsequence of) $\{u_i\}$ converges weakly in $W^{1,q}(D_{3/2}(0))$ for $q\in [1,2)$ and strongly in $W^{1,2}_{loc}(D_{3/2}(0)\setminus \Sigma)$ to a map $v\in C^{\infty}(D_{3/2}\setminus \Sigma,S^1)$ that is harmonic away from $\Sigma$.
\end{prop}

\hspace{3mm} The first and most important step in the proof of Theorem \ref{sec6thm} is contained in the following result, which describes the limiting measure explicitly when the maps $u_i$ converge to a limiting map $v$:

\begin{prop}\label{l1bdcase} Let $u_i\in W^{1,p_i}(D_2(0),S^1)$ be as in Proposition \ref{locdim2}, and suppose also that
$$\sup_i\|du_i\|_{L^1(D_2(0))}<\infty.$$
Passing to a subsequence, let $v$ be the limiting map $v=\lim_i u_i$ given by Proposition \ref{locdim2}. Then the limiting measure $\mu$ has the form
$$\mu=\Sigma_{a\in \Sigma}2\pi \deg(v,a)^2 \delta_a.$$
\end{prop}

For $p$-energy minimizers with respect to a fixed boundary condition, this result follows from the analysis of \cite{HL2}, in which case all of the degrees $\deg(v,a)$ are either $1$ or $-1$. It is also the immediate analog of the quantization result for 2-dimensional solutions of the Ginzburg-Landau equations in \cite{CM}, though the proof in our setting is much simpler. 

\hspace{3mm} The reason for the relative simplicity in our setting is the form of the Pohozaev identity. In \cite{CM}, on their way to demonstrating the quantization of the energy measures $\mu_{\epsilon}=\frac{|du_{\epsilon}(z)|^2}{|\log\epsilon|}dz$, Comte and Mironescu appeal to the quantization results of \cite{BBH} and \cite{BMR} for the potential measures $\frac{W(u_{\epsilon}(z))}{\epsilon^2}dz$. These quantization results--though by no means trivial--can be derived in a relatively straightforward way from a Pohozaev identity that relates the integral of $\frac{W(u_{\epsilon})}{\epsilon^2}$ on a disk to the behavior of $u_{\epsilon}$ on its boundary. It is then observed in \cite{CM} that the quantization of the potential measure gives strong constraints on the way that the degrees of the maps $u_{\epsilon}$ can vary around clusters of zeroes (or ``vortices") at different scales, which ultimately give rise to the quantization of the energy measures $\mu_{\epsilon}$.

\hspace{3mm} In our setting, the path is much simpler, because the normalized $p$-energy $(2-p)|du|^p$ simultaneously plays the roles occupied by the energy \emph{and} potential measures in the Ginzburg-Landau setting. In particular, we have the following nice Pohozaev-type identity:

\begin{lem}\label{pohozaev} Let $u\in W^{1,p}(D_2(0),S^1)$ be stationary $p$-harmonic on $D_2(0)$. On any annulus
$$A_{r_1,r_2}(a)=D_{r_2}(0)\setminus D_{r_1}(0)\subset D_2(0),$$
we then have
$$\int_{r_1}^{r_2}\left(\int_{D_r(a)}(2-p)|du|^p\right)dr=\int_{A_{r_1,r_2}(a)}|z-a|(|du|^p-p|du|^{p-2}|du(\frac{z-a}{|z-a|})|^2)dz.$$
\end{lem}

\begin{proof} The identity is simply a repackaging of the monotonicity formula in dimension two: By testing the inner variation equation 
$$\int |du|^pdiv(X)-p|du|^{p-2}\langle du^*du,\nabla X\rangle=0$$
against vector fields of the form $X(z)=\psi(z)(z-a)$ for test functions $\psi\in C_c^{\infty}(D_r(a))$ approximating the characteristic function $\chi_{D_r(a)}$, we find that
\begin{equation}
\int_{D_r(a)} (2-p)|du|^p=r\int_{\partial D_r(a)}(|du|^p-p|du|^{p-2}|du(\frac{z-a}{|z-a|})|^2
\end{equation}
for almost every $r\in [r_1,r_2]$. Integrating over $[r_1,r_2]$ then gives the desired equation.
\end{proof}

With this identity in hand, we can now argue in the spirit of \cite{BBH},\cite{BMR} to prove Proposition \ref{l1bdcase}:

\begin{proof}{(Proof of Proposition \ref{l1bdcase})}

\hspace{3mm} Let 
$$\Sigma=\{a_1,\ldots,a_k\},$$ 
so that the limiting map $v(z)$ satisfies
$$d^*(jv)=0\text{ and }T(jv)=\Sigma_{\ell=1}^k2\pi \kappa_{\ell}\delta_{a_{\ell}},$$
where $\kappa_{\ell}=\deg(v,a_{\ell})$ denotes the degree of $v$ about $a_{\ell}$. Letting $\bar{v}$ be the map given by
$$\bar{v}(z):=\Pi_{\ell=1}^k\left(\frac{z-a_{\ell}}{|z-a_{\ell}|}\right)^{\kappa_{\ell}},$$
we observe that 
$$d^*j\bar{v}=0\text{ and }T(j\bar{v})=T(jv),$$
so the difference $jv-j\bar{v}$ is strongly harmonic. In particular, it follows that
\begin{equation}
v=e^{i\varphi}\bar{v}
\end{equation}
for some harmonic function $\varphi \in C^{\infty}(D_2(0))$. 

\hspace{3mm} Now, set 
$$\delta_0:=\min\{|a_{\ell}-a_m|\mid 1\leq \ell<m\leq k\},$$
so that the density $\theta_{\ell}$ of $\mu$ at $a_{\ell}$ is given by
$$\mu(D_r(a_{\ell}))=\theta_{\ell}$$
for every $r\in (0,\delta_0)$. For any $\delta \in (0,\delta_0)$, it then follows from Lemma \ref{pohozaev} that
\begin{eqnarray*}
\theta_{\ell}&=&\frac{2}{\delta}\int_{\delta/2}^{\delta}\mu(D_r(a_{\ell}))dr\\
&=&\frac{2}{\delta}\lim_{i\to\infty}\int_{\delta/2}^{\delta}\mu_i(D_r(a_{\ell}))dr\\
&=&\frac{2}{\delta}\lim_{i\to\infty}\int_{A_{\delta/2,\delta}(a_{\ell})}|z-a_{\ell}|(|du_i|^p-p_i|du_i|^{p_i-2}|du_i(\frac{z-a_{\ell}}{|z-a_{\ell}|})|^2)dz.
\end{eqnarray*}
On the other hand, we also know that $u_i\to v$ strongly in $W^{1,2}(A_{\delta/2,\delta}(a_{\ell}))$ and $\|du_i\|_{L^{\infty}(A_{\delta/2,\delta}(a_{\ell}))}$ is uniformly bounded as $i\to\infty$,
so it follows that
\begin{eqnarray*}
\theta_{\ell}&=&\frac{2}{\delta}\lim_{i\to\infty}\int_{A_{\delta/2,\delta}(a_{\ell})}|z-a_{\ell}|(|du_i|^p-p_i|du_i|^{p_i-2}|du_i(\frac{z-a_{\ell}}{|z-a_{\ell}|})|^2)dz\\
&=&\frac{2}{\delta}\int_{A_{\delta/2,\delta}(a_{\ell})}|z-a_{\ell}|(|dv|^2-2|dv(\frac{z-a_{\ell}}{|z-a_{\ell}|})|^2dz.
\end{eqnarray*}

\hspace{3mm} Since $v=e^{i\varphi}\bar{v}$, we can expand $dv$ as
$$dv(z)=v(z)\cdot (id\varphi+\Sigma_{\ell=1}^k\kappa_{\ell}\frac{\overline{z-a_{\ell}}}{|z-a_{\ell}|^2}P_{z-a_{\ell}}^{\perp}),$$
where for $0\neq w\in \mathbb{R}^2$ we denote by $P_w^{\perp}$ projection onto the line perpendicular to $w$. In particular, if $z\in D_{\delta}(a_{\ell})$ for $\delta<\frac{\delta_0}{2}$, then $|z-a_m|>\delta_0-\delta>\frac{\delta_0}{2}$ for every $m\neq \ell$, and it follows that
$$|dv(z)-\kappa_{\ell}\frac{\overline{z-a_{\ell}}}{|z-a_{\ell}|^2}P_{z-a_{\ell}}^{\perp}|\leq \|d\varphi\|_{L^{\infty}}+\Sigma_{m\neq \ell}\frac{2|\kappa_m|}{\delta_0}=:K.$$
Combining this with the obvious estimate
$$|dv(z)|\leq \frac{K'}{|z-a_{\ell}|}\text{ on }D_{\delta}(a_{\ell})$$
(where $K'$ of course depends on $v$), we see that, on $D_{\delta}(a_{\ell})$,
\begin{equation}
||dv(z)|^2-\frac{\kappa_{\ell}^2}{|z-a_{\ell}|^2}|\leq \frac{K''}{|z-a_{\ell}|}
\end{equation}
and
\begin{equation}
|dv(\frac{z-a_{\ell}}{|z-a_{\ell}|})|^2\leq \frac{K''}{|z-a_{\ell}|}.
\end{equation}
In particular, on the annulus $A_{\delta/2,\delta}(a_{\ell})$, since
$$\int_{A_{\delta/2,\delta}(a_{\ell})}|z-a_{\ell}|\frac{\kappa_{\ell}^2}{|z-a_{\ell}|^2}=\pi\kappa_{\ell}^2\delta,$$
we can apply the preceding estimates to our computation of $\theta_{\ell}$ to conclude that
\begin{eqnarray*}
|\theta_{\ell}-2\pi\kappa_{\ell}^2|&=&|\frac{2}{\delta}\int_{A_{\delta/2,\delta}(a_{\ell})}|z-a_{\ell}|(|dv|^2-\frac{\kappa_{\ell}^2}{|z-a_{\ell}|^2}-2|dv(\frac{z-a_{\ell}}{|z-a_{\ell}|})|^2)dz\\
&\leq & \frac{2}{\delta}\int_{A_{\delta/2,\delta}(a_{\ell})}|z-a_{\ell}|\frac{K''}{|z-a_{\ell}|}\\
&=&2\pi K'' \delta.
\end{eqnarray*}
Since $\delta>0$ was arbitrary, it follows finally that
$$\theta_{\ell}=2\pi \kappa_{\ell}^2,$$
which is precisely what we wanted to show.
\end{proof}

Combining Proposition \ref{l1bdcase} with a simple contradiction argument, and scaling, we can formulate the following lemma:

\begin{cor}\label{l1bdquant} For any $\gamma>0$ and $\Lambda<\infty$, there exists $q(\gamma,\Lambda)\in (1,2)$ such that if $p>q$, and $u\in W^{1,p}(D_{2r}(x),S^1)$ is stationary $p$-harmonic with
\begin{equation}
(2-p)\theta_p(u,x,2r)+\frac{1}{r}\int_{D_{2r}(x)}|du|\leq \Lambda
\end{equation}
and $Sing(u)\cap D_{2r}(x)\subset D_r(x)$, then
\begin{equation}
dist((2-p)\theta_p(u,x,r),2\pi\mathbb{Z})<\gamma.
\end{equation}
\end{cor}

\hspace{3mm} We now remove the requirement of a uniform $W^{1,1}$ bound by arguing that, in the general setting of Theorem \ref{sec6thm}, the normalized energy measures $\mu_i$ are negligible on the complement of a collection of disks satisfying the conditions of Corollary \ref{l1bdquant}.

\begin{proof}{(Proof of Theorem \ref{sec6thm})}

\hspace{3mm} Let $u_i\in W^{1,p_i}(D_2(0),S^1)$ be a sequence of stationary $p_i$-harmonic maps as given, with
$$\mu_i\to \theta\delta_0,$$
and
$$E_{p_i}(u_i,D_2(0))\leq \frac{\Lambda}{2-p_i}.$$
We consider now a local version of the Hodge decomposition of Section \ref{hodgeests}; choosing some cutoff function $\chi \in C_c^{\infty}(D_{7/4}(0))$ such that $\chi\equiv 1$ on $D_{5/4}(0)$, we define
$$\xi_i:=-*\Delta^{-1}(\chi T(u_i))=-\langle (\chi T(u_i))(y),G(x,y)\rangle dx^1\wedge dx^2$$
and
$$\varphi_i:=\Delta^{-1}(\chi div(ju_i))=\langle \chi div([1-|ju_i|^{p-2}]ju_i)(y),G(x,y)\rangle,$$
where $G(x,y)=\frac{-1}{2\pi}\log|x-y|$ is the two-dimensional Green's function. 

\hspace{3mm} Writing
$$h_i=ju_i-d^*\xi_i-d\varphi_i,$$
we see that (distributionally)
$$div(h_i)=(1-\chi)div(ju_i)=0$$
and
$$d^*dh_i=d^*([1-\chi]T(u_i)).$$
In particular, $h_i$ is harmonic on the disk $D_{5/4}(0)$ (where $\chi\equiv 1$), and it follows that
\begin{equation}\label{obvhlinfty}
\|h_i\|_{L^{\infty}(D_{1/2}(0))}\leq C\|h_i\|_{L^{p_i}(D_1(0)\setminus D_{3/4}(0))}.
\end{equation}

\hspace{3mm} Next, consider the $(2-p_i)^{1/p_i}$-neighborhoods
$$U_i:=\mathcal{N}_{(2-p_i)^{1/p_i}}(Sing(u_i))$$
about the singular sets $Sing(u_i)$. Our goal now is to show that 
\begin{equation}\label{noengoffu}
\lim_{i\to\infty} \mu_i(D_{1/2}(0)\setminus U_i)=0,
\end{equation}
and that $U_i$ is contained in a finite union of disks satisfying the hypotheses of Corollary \ref{l1bdquant}. To this  end, we first observe that, by Corollary \ref{tubdsscaled} and Lemma \ref{wgradestlem} of the appendix, $\nabla\xi$ satisfies the pointwise bound
\begin{equation}
|\nabla \xi_i|(x)\leq \frac{C}{dist(x,Sing(u_i))}
\end{equation}
for $x\in D_1(0)$. Putting this together with the volume estimate of Lemma \ref{singvollem},
we find that
\begin{eqnarray*}
\int_{D_{1/2}(0)\setminus U_i}|\nabla \xi_i|^{p_i}&\leq&\Sigma_{j=1}^{\lfloor \frac{1}{p_i}|\log_2(2-p_i)|\rfloor}\int_{\mathcal{N}_{2^{-j}}(Sing(u_i))\setminus \mathcal{N}_{2^{-j-1}}(Sing(u_i))}|du_i|^{p_i}\\
&\leq &\Sigma_{j=1}^{\lfloor \frac{1}{p_i}|\log_2(2-p_i)|\rfloor}C2^{-jp_i}\cdot 2^{(j+1)p_i},
\end{eqnarray*}
and therefore
\begin{equation}\label{gradxioffu}
\int_{D_{1/2}(0)\setminus U_i}|\nabla\xi_i|^{p_i}\leq C |\log(2-p_i)|.
\end{equation}

\hspace{3mm} For $d\varphi_i$, the arguments in the proofs of Proposition \ref{hodgeglob} and Lemma \ref{offsingest} again yield the local estimates
\begin{equation}\label{phiw1qloc}
\|\varphi_i\|_{W^{1,q}(D_1(0))}\leq C(q)(2-p_i)^{1-1/p_i}|\log(2-p_i)|
\end{equation}
for $q\in (1,p_i)$ and
\begin{equation}\label{firstpassdphi}
\|(2-p_i)^{1/p_i}d\varphi_i\|^2_{L^2(D_1(0)\setminus U_i)}\leq C (2-p_i)^{2-2/p_i}|\log(2-p_i)|^2,
\end{equation}
respectively. In particular, rearranging (\ref{firstpassdphi}) and recalling that $(2-p_i)^{p_i-2}$ is uniformly bounded as $i\to\infty$, we see that
\begin{equation}\label{dphioffu}
\int_{D_1(0)\setminus U_i}|d\varphi_i|^2\leq C |\log(2-p_i)|^2.
\end{equation}
Now, to estimate $h_i$, we observe that
\begin{eqnarray*}
\|h_i\|_{L^{p_i}(D_1(0)\setminus D_{3/4}(0))}&\leq &\|du_i\|_{L^{p_i}(D_1(0)\setminus D_{3/4}(0))}+\|d^*\xi_i\|_{L^{p_i}(D_1(0)\setminus D_{3/4}(0))}\\
&&+\|d\varphi_i\|_{L^{p_i}(D_1(0)\setminus D_{3/4}(0))}\\
&\leq &(2-p_i)^{-1/p_i}\mu_i(D_1(0)\setminus D_{3/4}(0))^{1/p_i}+\|d^*\xi_i\|_{L^{p_i}(D_1(0)\setminus U_i)}\\
&&+\|d\varphi_i\|_{L^{p_i}(D_1(0)\setminus D_{3/4}(0))}\\
&\leq &(2-p_i)^{-1/p_i}\mu_i(D_1(0)\setminus D_{3/4}(0))^{1/p_i}+C|\log(2-p_i)|^{2/p_i}.
\end{eqnarray*}
Since $\mu=\lim\mu_i$ vanishes on compact subsets of $D_1(0)\setminus \{0\}$ by assumption, it then follows that
$$\|h_i\|_{L^{p_i}(D_1(0)\setminus D_{3/4}(0))}=o(\frac{1}{2-p_i});$$
and by (\ref{obvhlinfty}), we therefore have
\begin{equation}\label{hlinfty2}
\|h\|^{p_i}_{L^{\infty}(D_{1/2}(0))}\leq \frac{\delta_i}{2-p_i},
\end{equation}
where $\lim_{i\to\infty}\delta_i=0$.

\hspace{3mm} Putting together the estimates (\ref{gradxioffu}), (\ref{dphioffu}), and (\ref{hlinfty2}), we see finally that
\begin{eqnarray*}
\lim_{i\to\infty}\mu_i(D_{1/2}(0)\setminus U_i)&=&\lim_{i\to\infty}(2-p_i)\int_{D_{1/2}(0)\setminus U_i}|du_i|^{p_i}\\
&\leq &C\limsup_{i\to\infty}\int_{D_{1/2}(0)\setminus U_i}(2-p_i)[|d^*\xi_i|^{p_i}+|d\varphi_i|^{p_i}+|h_i|^{p_i}]\\
&\leq &C\limsup_{i\to\infty} ((2-p_i)|\log(2-p_i)|^2+\delta_i)\\
&=&0,
\end{eqnarray*}
confirming (\ref{noengoffu}).

\hspace{3mm} Next, as in the proof of Lemma \ref{singvollem}, we know from a simple Vitali covering argument that 
$$U_i\subset \bigcup_{\ell=1}^{k_i}D_{3(2-p_i)^{1/p_i}}(x^i_{\ell})$$ 
for some $x^i_1,\ldots,x^i_{k_i}\in Sing(u_i)$ such that 
$$D_{(2-p_i)^{1/p_i}}(x^i_{\ell})\cap D_{(2-p_i)^{1/p_i}}(x^i_m)=\varnothing\text{ when }\ell\neq m,$$
and it follows from Proposition \ref{sharpdensbd} that 
$$[(2-p_i)^{1/p_i}]^{p_i-2}\mu_i(U_i)\geq 2\pi k_i.$$
In particular, $k_i$ is uniformly bounded independent of $i$, so passing to a subseqence, we can take $k_i=k$ to be constant. Moreover, setting
$$\alpha_{\ell m}^i:=(2-p_i)^{-1/p_i}dist(x^i_{\ell},x^i_m),$$
we can pass to a further subsequence for which the (possibly infinite) limits
$$\alpha_{\ell m}=\lim_{i\to\infty}\alpha_{\ell m}^i$$
exist. Relabeling indices if necessary, there is then some $m_0\leq k$ such that 
$$\alpha_{\ell m}=\infty\text{ for }1\leq \ell<m\leq m_0,$$
and for every $m>m_0$, there is some $1\leq \ell\leq m_0$ for which $\alpha_{\ell m}<\infty$. 

\hspace{3mm} Now, let
$$A:=\max(\{3\} \cup\{2\alpha_{\ell m}\mid \alpha_{\ell m}<\infty\}),$$
$$r_i:=A(2-p_i)^{1/p_i},$$
and for $1\leq \ell \leq m_0$, define the disks
$$D_{i,\ell}:=D_{r_i}(x_{\ell}^i).$$
For $i$ sufficiently large, we then see that
$$U_i\subset D_{i,1}\cup\cdots\cup D_{i,m_0},\text{\hspace{3mm} }D_{i,\ell}\cap D_{i,m}=\varnothing\text{ if }\ell\neq m,$$
and 
$$Sing(u_i)\cap D_{2r_i}(x_{\ell}^i)\subset D_{r_i}(x^i_{\ell}).$$
In particular, since 
\begin{equation}\label{obvengbd}
\limsup_{i\to\infty}(2-p_i)\theta_{p_i}(u_i,x_{\ell}^i,2r_i)\leq 2\theta<\infty
\end{equation}
by the monotonicity formula, our disks will satisfy the conditions of Corollary \ref{l1bdquant} for some $\Lambda$, once we show that
\begin{equation}\label{wantl1bd}
\limsup_{i\to\infty}r_i^{-1}\|du_i\|_{L^1(D_{2r_i}(x_{\ell}^i))}<\infty.
\end{equation}

\hspace{3mm} To establish (\ref{wantl1bd}), we consider separately the components $h_i$, $d\varphi_i$, and $d^*\xi_i$ of the local Hodge decomposition. For $h_i$, we have seen already that
$$\|h_i\|_{L^{\infty}}\leq \delta_i^{1/p_i}(2-p_i)^{-1/p_i}\leq \delta_i^{1/p_i}Ar_i^{-1},$$
where $\delta_i\to 0$, so that
\begin{equation}
r_i^{-1}\int_{D_{2r_i}(x_{\ell}^i)}|h_i|\leq \delta_i^{1/p_i} Ar_i^{-2}\cdot 4\pi r_i^2\leq 1
\end{equation}
for $i$ sufficiently large. For $d\varphi_i$, recall that, for $q<p_i$,
$$\|d\varphi_i\|_{L^q(D_1(0))}\leq C(q)(2-p_i)^{1-1/p_i}|\log(2-p_i)|,$$
so that
\begin{eqnarray*}
r_i^{-1}\int_{D_{2r_i}(x_{\ell}^i)}|d\varphi_i|&\leq &r_i^{-1}\|d\varphi_i\|_{L^q(D_1(0))}(4\pi r_i^2)^{1-1/q}\\
&\leq & r_i^{-1}C(q)(2-p_i)^{1-1/p_i}|\log(2-p_i)|r_i^{2-2/q}\\
&\leq &C(q) A^{1-2/q}(2-p_i)^{1-2/(p_iq)}|\log(2-p_i)|.
\end{eqnarray*}
For $i$ sufficiently large, we can take $p_i>q=\frac{3}{2}$ in the estimate above, to obtain
\begin{eqnarray*}
\limsup_{i\to\infty} r_i^{-1}\int_{D_{2r_i}(x^i_{\ell})}|d\varphi_i|&\leq & C(q)A^{1-2/q}\lim_{i\to\infty}(2-p_i)^{1-8/9}|\log(2-p_i)|\\
&=&C\lim_{i\to\infty}(2-p_i)^{1/9}|\log(2-p_i)|\\
&=&0.
\end{eqnarray*}

\hspace{3mm} Next, employing the pointwise gradient estimate (\ref{gradxioffu}) for $\xi_i$ with Lemma \ref{singvollem}, we see that
\begin{eqnarray*}
r_i^{-1}\int_{D_{2r_i}(x^i_{\ell})}|d^*\xi_i|&\leq &r_i^{-1}\int_{\mathcal{N}_{2r_i}(Sing(u_i))}|d^*\xi|\\
&=&r_i^{-1}\Sigma_{j=0}^{\infty}\int_{\mathcal{N}_{2^{1-j}r_i}(Sing(u_i))\setminus \mathcal{N}_{2^{-j}r_i}(Sing(u_i))}|d^*\xi|\\
\text{( by (\ref{gradxioffu}) )}&\leq &r_i^{-1}\Sigma_{j=0}^{\infty}\int_{\mathcal{N}_{2^{1-j}r_i}(Sing(u_i))\setminus \mathcal{N}_{2^{-j}r_i}(Sing(u_i))}\frac{C}{2^{-j}r_i}\\
\text{( by Lemma \ref{singvollem} )}&\leq &r_i^{-1}\Sigma_{j=0}^{\infty}\frac{C}{2^{-j}r_i}\cdot C(2^{1-j}r_i)^{p_i}\\
&\leq & Cr_i^{p_i-2}\Sigma_{j=0}^{\infty}(2^{1-p_i})^j.
\end{eqnarray*}
And since
$$\Sigma_{j=0}^{\infty}(2^{1-p_i})^j=\frac{2^{p_i}}{2^{p_i}-2}\to 2$$
and
$$r_i^{p_i-2}=A^{p_i-2}(2-p_i)^{\frac{p_i-2}{p_i}}\to 1$$
as $i\to\infty$, it follows that
\begin{equation}\label{almostthere}
\limsup_{i\to\infty}r_i^{-1}\int_{D_{2r_i}(x^i_{\ell})}|d^*\xi_i|<\infty.
\end{equation}
Combining this with the preceding estimates for $h_i$ and $d\varphi_i$, we see that (\ref{wantl1bd}) indeed holds.

\hspace{3mm} Finally, letting
$$\Lambda:=1+\max_{1\leq \ell \leq m_0}\limsup_{i\to\infty}[(2-p_i)\theta_{p_i}(u_i,x_{\ell}^i,2r_i)+r_i^{-1}\int_{D_{2r_i}(x_{\ell}^i)}|du_i|],$$
and choosing an arbitrary $\gamma>0$, it follows from Corollary \ref{l1bdquant} that for $i$ sufficiently large,
\begin{equation}
dist((2-p_i)\theta_{p_i}(u_i,x_{\ell}^i,r_i),2\pi \mathbb{Z})<\gamma;
\end{equation}
in particular, we deduce that
\begin{equation}\label{densnearint}
\limsup_{i\to\infty}dist((2-p_i)\Sigma_{\ell=1}^{m_0}\theta_{p_i}(u_i,x^i_{\ell},r_i),2\pi\mathbb{Z})=0.
\end{equation}
Now, by the disjointness of the disks $\{D_{i,\ell}\}_{\ell=1}^{m_0}$, we know that
$$(2-p_i)\Sigma_{\ell=1}^{m_0}\theta_{p_i}(u_i,x^i_{\ell},r_i)=r_i^{p_i-2}\mu_i(\bigcup_{\ell=1}^{m_0}D_{i,\ell}),$$
and since $\lim_{i\to\infty}r_i^{p_i-2}=1$, it follows that
$$\lim_{i\to\infty}(2-p_i)\Sigma_{\ell=1}^{m_0}\theta_{p_i}(u_i,x^i_{\ell},r_i)=\lim_{i\to\infty}\mu_i(\bigcup_{\ell=1}^{m_0}D_{i,\ell}).$$
On the other hand, since the disks $D_{i,\ell}$ cover $U_i$, we know from (\ref{noengoffu}) that
\begin{eqnarray*}
\lim_{i\to\infty}\mu_i(D_{1/2}(0))&=&\lim_{i\to\infty} (\mu_i(D_{1/2}(0)\setminus \bigcup_{\ell} D_{i,\ell})+\mu_i(\bigcup_{\ell}D_{i,\ell}))\\
&=&\lim_{i\to\infty}\mu_i(\bigcup_{\ell}D_{i,\ell})\\
&=&\lim_{i\to\infty}(2-p_i)\Sigma_{\ell=1}^{m_0}\theta_{p_i}(u_i,x^i_{\ell},r_i).
\end{eqnarray*}
By (\ref{densnearint}), it then follows that
$$\theta=\lim_{i\to\infty}\mu_i(D_{1/2}(0))\in 2\pi\mathbb{N},$$
as desired.
\end{proof}

\begin{remark} In higher dimensions, one would like to show, analogously, that for a sequence $u_i\in W^{1,p_i}(B_2^n(0),S^1)$ of stationary $p_i$-harmonic maps with energy concentrating along an $(n-2)$-plane $\mathcal{L}^{n-2}$, the limiting measure $\mu$ must have the form
$$\mu=2\pi m\cdot\mathcal{H}^{n-2}\lfloor \mathcal{L}$$
for some $m\in \mathbb{N}$. As in the proof of Theorem \ref{sec6thm}, it is possible to reduce the problem to the case where the maps $u_i$ converge away from $\mathcal{L}$, but for the moment we have no higher-dimensional analog of Proposition \ref{l1bdcase}. 

\hspace{3mm} In particular, if one na\"{i}vely attempts to generalize the Pohozaev identity to this setting (for instance, by testing the component of the position vector field perpendicular to $\mathcal{L}$ in the inner variation equation), the $\mathcal{L}$ components $|du_i(\mathcal{L})|$ of the derivatives of $u_i$ invariably get in the way. And while it is easy to see that $(2-p_i)\int|du_i(\mathcal{L})|^{p_i}\to 0$ under these assumptions, one would need the much stronger vanishing $\int |du_i(\mathcal{L})|^{p_i}\to 0$ for the arguments of this section to work. We are nonetheless optimistic about the prospect of extending the integrality result to higher dimensions, but a proof will undoubtedly require some interesting new ideas. 

\end{remark}

\section{Natural Min-Max Constructions}\label{sec7}

\subsection{Generalized Ginzburg-Landau Functionals}\hfill

\hspace{3mm} Let $M^n$ and $N$ be compact Riemannian manifolds, with $N$ isometrically embedded in some Euclidean space $\mathbb{R}^L$. For $1<p\leq n$ and $\epsilon>0$, Wang studies in \cite{Wang} the generalized Ginzburg-Landau functionals 
$$E_{p,\epsilon}: W^{1,p}(M,\mathbb{R}^L)\to\mathbb{R}$$
given by
$$E_{p,\epsilon}(u)=\int_M(|du|^p+\epsilon^{-p}F(u)),$$
where the function $F:\mathbb{R}^L\to \mathbb{R}$ has the form $F(y)=\lambda(dist(y,N)^2)$ for a function $\lambda\in C^{\infty}(\mathbb{R})$ satisfying $\lambda'\geq 0$, $\lambda''\geq 0$. 
$$\lambda(t)=t\text{ for }t\leq \delta_N^2,\text{ and }\lambda(t)=4\delta_N^2\text{ for }t\geq 4\delta_N^2,$$
(Here, $\delta_N>0$ is chosen such that nearest-point projection to $N$ is well-defined and smooth on the $2\delta_N$-neighborhood of $N$ \cite{Wang}.) Thus, as $\epsilon\to 0$, the potential term in the energies $E_{p,\epsilon}$ penalizes deviation from the target manifold $N$, while for $N$-valued maps $u$, one simply recovers the $p$-energy $E_{p,\epsilon}(u)=E_p(u)$.

\hspace{3mm} For $p=2$, the $\epsilon\to 0$ asymptotics of bounded-energy critical points and negative gradient flows of these functionals had previously been studied in \cite{CL}, \cite{CS}, \cite{LW} as regularizations of harmonic maps and harmonic map heat flows. While in the $p=2$ setting one encounters the familiar bubbling phenomena that arise in the study of harmonic maps, for $p\notin\mathbb{N}$, Wang demonstrates that (much like Proposition \ref{strongcomp} for $p$-harmonic maps), bounded-energy sequences of critical points enjoy a strong compactness property:
\begin{thm}\label{wangthm}{$($\cite{Wang} \emph{Theorem A, Corollary B}$)$} If $p\in (1,n)\setminus\mathbb{N}$, and $\{u_{\epsilon_i}\}$ is a sequence of critical points for $E_{p,\epsilon_i}$ with $\epsilon_i\to 0$ and
\begin{equation}
\sup_i E_{p,\epsilon_i}(u_{\epsilon_i})<\infty,
\end{equation}
then a subsequence of $\{u_{\epsilon_i}\}$ converges strongly in $W^{1,p}(M,\mathbb{R}^L)$ to a stationary $p$-harmonic map $u\in W^{1,p}(M,N)$.
\end{thm}
As a consequence, for $p\in (1,n)\setminus\mathbb{N}$, the functionals $E_{p,\epsilon}$ are naturally suited to the construction of stationary $p$-harmonic maps via min-max methods, in light of the following elementary fact:

\begin{lem}\label{glps}
The generalized Ginzburg-Landau energy $E_{p,\epsilon}$ is a $C^1$ functional on $W^{1,p}(M,\mathbb{R}^L)$, with derivative
$$\langle E_{p,\epsilon}'(u),v\rangle=\int_M p\langle |du|^{p-2}du,dv\rangle+\epsilon^{-p}\langle DF(u),v\rangle,$$
and satisfies the following Palais-Smale condition:
if $u_j\in W^{1,p}(M,\mathbb{R}^L)$ is a sequence satisfying
\begin{equation}\label{psbdd}
\sup_j\|u_j\|_{W^{1,p}}\leq C<\infty
\end{equation}
and
\begin{equation}\label{pscs}
\lim_{j\to\infty}\|E_{p,\epsilon}'(u)\|_{(W^{1,p})^*}=0,
\end{equation}
then $\{u_j\}$ has a subsequence that converges strongly in $W^{1,p}$.
\end{lem}

\begin{proof} The first statement is trivial. The proof of the Palais-Smale condition is also quite standard, but we include it for completeness: 

For a sequence $\{u_j\}$ satisfying (\ref{psbdd}), we know from Rellich's theorem that a subsequence (which we continue to denote by $\{u_j\}$) converges weakly in $W^{1,p}$ and strongly in $L^p$ to a limiting function $u\in W^{1,p}.$ To confirm that the convergence is also strong in $W^{1,p}$, it is enough to show that
\begin{equation}\label{normlimclm}
\|du\|_{L^p}\geq\limsup_j\|du_j\|_{L^p}.
\end{equation}
And indeed, if the $\{u_j\}$ also satisfies (\ref{pscs}), then we see that
\begin{eqnarray*}
0&=&\lim_{j\to\infty}\|u_j-u\|_{W^{1,p}}\|E_{p,\epsilon}'(u_j)\|_{(W^{1,p})^*}\\
&\geq &\limsup_{j\to\infty}\langle E_{p,\epsilon}'(u_j),u_j-u\rangle\\
&=&\limsup_{j\to\infty} p\int_M(|du_j|^p-\langle |du_j|^{p-2}du_j,du\rangle)\\
&&+\limsup_{j\to\infty} \int_M\langle DF(u_j),u_j-u\rangle\\
\text{( since $u_j\to u$ strongly in $L^p$ )}&=&p\limsup_{j\to\infty} \int_M(|du_j|^p-\langle |du_j|^{p-2}du_j,du\rangle)\\
&\geq &\limsup_{j\to\infty}\|du_j\|_{L^p}^p-\|du_j\|_{L^p}^{p-1}\|du\|_{L^p},
\end{eqnarray*}
from which (\ref{normlimclm}) follows, completing the proof.
\end{proof}

\hspace{3mm} In the remainder of this section, we employ a simple min-max constructions for the energies $E_{p,\epsilon}$ for $N=S^1$ and $p\in (1,2)$, with arguments very similar to those of \cite{Ste1}, to prove Theorem \ref{existthm} of the introduction. Though we focus here on $p$-harmonic maps to $S^1$ with $p\in (1,2)$, we remark that the following construction (though not the specific energy bounds) can be generalized to produce nontrivial stationary $p$-harmonic maps from manifolds of dimension $n\geq k$ into arbitrary targets $N$ with $\pi_{k-1}(N)\neq \varnothing$, for $p\in (1,k)\setminus \mathbb{N}$. A detailed study of this general construction is beyond the scope of this paper, but may be an interesting direction for further investigation.

\subsection{The Saddle Point Construction and an Upper Bound for the Energies}\hfill

\hspace{3mm} We restrict ourselves now to the case where our target manifold $N$ is $S^1$, embedded in $\mathbb{R}^2$ as the boundary of the unit disk $D_1(0)$, and consider the collection $\Gamma_p(M)$ of two-parameter families $y\mapsto h_y\in W^{1,p}(M,\mathbb{R}^2)$ given by
\begin{equation}
\Gamma_p(M):=\{h\in C^0(D_1^2,W^{1,p}(M,\mathbb{R}^2))\mid h_y\equiv y\text{ for }y\in S^1\}.
\end{equation}
For $p\in (1,2)$ and $\epsilon>0$, we define the min-max energy levels $c_{p,\epsilon}$ by
\begin{equation}
c_{p,\epsilon}(M):=\inf_{h\in \Gamma_p(M)}\max_{y\in D_1^2}E_{p,\epsilon}(h_y),
\end{equation}
and the limiting energy levels
\begin{equation}
c_p(M):=\sup_{\epsilon>0}c_{p,\epsilon}(M)=\lim_{\epsilon\to 0}c_{p,\epsilon}(M).
\end{equation}

\hspace{3mm} Now, we've observed that $E_{p,\epsilon}$ is a $C^1$ functional on $W^{1,p}(M,\mathbb{R}^2)$, which evidently vanishes on the circle of constant maps to $S^1$. Thus, if we can show that $c_{p,\epsilon}(M)>0,$
then we can apply the Saddle Point Theorem of Rabinowitz (see, e.g., Chapter 3 of \cite{Gh}, Chapter 4 of \cite{Rab}) to conclude that for any minimizing sequence of families
$$h^j\in \Gamma_p(M),\text{\hspace{3mm} }\max_{y\in D_1^2}E_{p,\epsilon}(h^j_y)\to c_{p,\epsilon}(M),$$
there exists a sequence $v_j\in W^{1,p}(M,\mathbb{R}^2)$ for which
\begin{equation}\label{psseq}
\lim_{j\to\infty}E_{p,\epsilon}(v_j)=c_{p,\epsilon}(M),\text{ }\lim_{j\to\infty}\|E_{p,\epsilon}'(u)\|_{(W^{1,p})^*}=0,
\end{equation}
and
\begin{equation}\label{mmseqbdd}
\lim_{j\to\infty}dist_{W^{1,p}}(v_j,h^j(D_1^2))=0.
\end{equation} 
Moreover, since we can deform any minimizing sequence of families $h^j$ to one whose maps take values in the unit disk by applying a retraction, we can obtain in this way a sequence $v_j$ satisfying (\ref{psseq}) and a uniform bound
$$\sup_j\|v_j\|_{W^{1,p}}<\infty.$$
In particular, it will then follow from Lemma \ref{glps} that there is indeed a critical point $u_{p,\epsilon}$ of $E_{p,\epsilon}$ with energy
$$E_{p,\epsilon}(u_{p,\epsilon})=c_{p,\epsilon}(M).$$

\hspace{3mm} To check that $c_{p,\epsilon}(M)>0$, we follow the same standard arguments as in \cite{Ste1}. Namely, for any $h\in \Gamma_p(M)$, we note that the averaging map
$$D_1^2\ni y \mapsto \frac{1}{Vol(M)}\int_Mh_y \in \mathbb{R}^2$$
defines a continuous map from $D_1^2\to \mathbb{R}^2$ which restricts to the identity on $S^1$, so that by elementary degree theory, there must be some $y_0\in D_1^2$ for which $\int_Mh_{y_0}=0$. Now, the $L^p$ Poincar\'{e} inequality gives us a constant $C_p(M)$ such that
$$\int_M |v|^p\leq C_p(M)\int_M|dv|^p$$
whenever $\int_Mv=0$, so by the preceding observation, for any $h\in \Gamma_p(M)$, there is some $y_0\in D_1^2$ for which $v=h_{y_0}$ satisfies
\begin{eqnarray*}
E_{p,\epsilon}(v)&\geq &C_p(M)^{-1}\int_M|v|^p+\int_M\epsilon^{-p}F(v)\\
&=&C_p(M)^{-1}\int_{\{|v|\geq \frac{1}{2}\}}|v|^p+\int_{\{|v|\leq \frac{1}{2}\}}\epsilon^{-p}F(v)\\
&\geq &C_p(M)^{-1}2^{-p}Vol(\{|v|\geq \frac{1}{2}\})+\epsilon^{-p}Vol(\{|v|< \frac{1}{2}\})\delta_0,
\end{eqnarray*}
where we've set $\delta_0:=\min\{F(y)\mid |y|\leq \frac{1}{2}\}>0$. In particular, since
$$Vol(\{|v|\geq \frac{1}{2}\})+Vol(\{|v|<\frac{1}{2})=Vol(M),$$
it follows that, for $\epsilon<1$,
$$\max_{y\in D_1^2}E_{p,\epsilon}(h_y)\geq E_{p,\epsilon}(h_{y_0})\geq \delta(p,M)>0$$
for any family $h\in \Gamma_p(M)$. Taking the infimum over $h\in \Gamma_p(M)$, we confirm that
$$c_{p,\epsilon}(M)\geq \delta(p,M)>0.$$

\hspace{3mm} In summary, we've so far established that
\begin{lem}\label{exist0} For $p\in (1,2)$ and $\epsilon\in (0,1)$, there exists a critical point $u_{p,\epsilon}\in W^{1,p}(M,\mathbb{R}^2)$ of $E_{p,\epsilon}$ satisfying
$$E_{p,\epsilon}(u_{p,\epsilon})=c_{p,\epsilon}(M)\geq \delta(p,M)>0.$$
\end{lem}

\hspace{3mm} Next, we will establish an upper bound for the limiting energy levels $c_p(M)=sup_{\epsilon>0}c_{p,\epsilon}(M)$: namely, we show that

\begin{claim}\label{cpupper} There exists $C(M)<\infty$ independent of $p$ such that
\begin{equation}
c_p(M)\leq \frac{C(M)}{2-p}.
\end{equation}
\end{claim}

The finiteness of $c_p(M)$ will then allow us to apply Theorem \ref{wangthm} to deduce the existence of a corresponding stationary $p$-harmonic map, while the boundedness of $(2-p)c_p(M)$ will provide the upper bound in Theorem \ref{existthm}.

\begin{proof} Again, the proof is very close to that of the analogous statement in Section 4 of \cite{Ste1}, albeit somewhat simpler, since in this case we can produce a single family $h$ lying in $\Gamma_p(M)$ for every $p\in (1,2)$ and satisfying a bound
\begin{equation}\label{wantfambd}
\max_{y\in D_1^2}E_{p,\epsilon}(h_y)\leq \frac{C(M)}{2-p}
\end{equation}
of the desired form.

\hspace{3mm} For $y\in D_1^2\setminus S^1$, define $v_y\in \bigcap_{p\in [1,2)}W^{1,p}(\mathbb{R}^2, S^1)$ by
\begin{equation}
v_y(z)=\frac{z+(1-|y|)^{-1}y}{|z+(1-|y|)^{-1}y|},
\end{equation}
and set $v_y\equiv y$ for $y\in S^1$. Fix also a triangulation of $M^n$--that is, choose a bi-Lipschitz map $\Phi: M\to |\mathcal{K}|$ from $M$ to the underlying space of a simplicial complex $\mathcal{K}$ in some Euclidean space $\mathbb{R}^L$. Applying a generic rotation, we can arrange that the projection map $P$ from $\mathbb{R}^L$ to the plane $\mathbb{R}^2\times 0$ has full rank on the $n$-dimensional subspace parallel to each $n$-dimensional simplex $\Delta\in \mathcal{K}$. Denoting by $f\in Lip(M,\mathbb{R}^2)$ the composition 
$$f(x)=(\Phi^1(x),\Phi^2(x))$$
of the projection $P:\mathbb{R}^L\to \mathbb{R}^2\times 0$ with $\Phi: M\to \mathbb{R}^L$, we define the family
\begin{equation}\label{hfamdef}
h_y:=v_y\circ f.
\end{equation}

\hspace{3mm} Our task now is to show that $y\mapsto h_y$ defines a continuous map $D_1^2\to W^{1,p}(M,\mathbb{R}^2)$, satisfying (\ref{wantfambd}). First, since $\Phi$ is bi-Lipschitz and $\mathcal{K}$ is finite, we observe that it is enough to establish this for the family
$$F_y:=v_y\circ P|_{\Delta}$$
on each $n$-dimensional face $\Delta \in\mathcal{K}$. And since we've also chosen $\mathcal{K}$ such that the restriction $P_{\Delta}$ of the projection map to $\Delta$ has full rank, we can write
$$P|_{\Delta}=P_0\circ L,$$
where $L:\Delta\to \Delta'\subset \mathbb{R}^n$ is an invertible affine-linear map, and $P_0: \mathbb{R}^n\to \mathbb{R}^2$ is simply the projection
$$P_0(x^1,\ldots,x^n)=(x^1,x^2)$$
onto the first two coordinates. In particular, it is enough to show that on a bounded domain $\Omega \subset \mathbb{R}^n$, the family
$$y\mapsto F_y(x)=v_y(x^1,x^2)$$
is continuous in $W^{1,p}$ for each $p\in (1,2)$, and satisfies
\begin{equation}\label{wantomegabd}
\max_{y\in D_1^2}E_{p,\epsilon}(F_y)\leq \frac{C_{\Omega}}{2-p}.
\end{equation}

\hspace{3mm} This is straightforward. By direct computation, the energy $E_{p,\epsilon}$ of $F_y$ on $\Omega$ satisfies
\begin{eqnarray*}
E_{p,\epsilon}(F_y)&=&\int_{(z,x')\in (\mathbb{R}^2\times \mathbb{R}^{n-2})\cap\Omega}|z+(1-|y|)^{-1}y|^{-p}dzdx'\\
&\leq &\int_{x'\in P_0(\Omega)}\frac{2\pi}{2-p}diam(\Omega)^{2-p}dx'\\
&\leq &\frac{C_ndiam(\Omega)^{n-p}}{2-p}.
\end{eqnarray*}
Thus, (\ref{wantomegabd}) holds, and we see moreover that the energy $y\mapsto E_{p,\epsilon}(F_y)$ varies continuously in $y$. Since the family $y\mapsto F_y$ is obviously weakly continuous in $W^{1,p}(\Omega,S^1)$, it follows that $y\mapsto F_y$ is strongly continuous as well. 

\hspace{3mm} We conclude finally that the families $y\mapsto h_y$ defined by (\ref{hfamdef}) indeed belong to $\Gamma_p(M)$, and satisfy
$$\max_{y\in D_1^2}E_{p,\epsilon}(h_y)\leq \frac{C(M)}{2-p},$$
where the constant $C(M)$ is determined by our choice of triangulation $\Phi:M\to |\mathcal{K}|$. In particular, it follows that
$$c_{p,\epsilon}(M)\leq \frac{C(M)}{2-p}$$
for every $\epsilon>0$, and taking the supremum over $\epsilon>0$, we therefore have 
\begin{equation}
c_p(M)\leq \frac{C(M)}{2-p},
\end{equation}
as desired.
\end{proof}

\hspace{3mm} Since $c_p(M)<\infty$, we can apply Theorem \ref{wangthm} to the min-max critical points $u_{p,\epsilon}$ of Lemma \ref{exist0}, to conclude that

\begin{prop}\label{exist1}
On every compact Riemannian manifold $M^n$ of dimension $n\geq 2$, there exists for every $p\in (1,2)$ a stationary $p$-harmonic map $u_p\in W^{1,p}(M,S^1)$ to $S^1$ of energy
$$0<E_p(u)=c_p(M)\leq \frac{C(M)}{2-p}.$$
\end{prop}

\subsection{Lower Bounds for $c_p(M)$}\hfill

\hspace{3mm} To complete the proof of Theorem \ref{existthm}, it remains to show that the min-max energies $c_p(M)$ satisfy a lower bound of the form
\begin{equation}\label{wantlowbd}
c_p(M)\geq \frac{c(M)}{2-p}.
\end{equation}
To achieve this, we argue as in Section 3 of \cite{Ste1}, with Proposition \ref{sharpdensbd} taking on the role played by the $\eta$-ellipticity theorem in the Ginzburg-Landau setting. 

\hspace{3mm} We begin by observing that (\ref{wantlowbd}) holds for the round sphere. As discussed in the proof of Proposition \ref{sharpdensbd}, since $b_1(S^n)=0$, every nontrivial weakly $p$-harmonic map $u\in W^{1,p}(S^n,S^1)$ must have singularities. In particular, the stationary $p$-harmonic maps $u_p$ of energy $c_p(M)>0$ constructed above must have nontrivial singular set, and from Proposition \ref{sharpdensbd} and monotonicity, it indeed follows that
\begin{equation}\label{spherelbd}
\liminf_{p\to 2}(2-p)c_p(S^n)=\liminf_{p\to 2}(2-p)E_p(u_p)>0.
\end{equation}
The estimate for arbitrary $M^n$ is then an easy consequence of the following claim:
\begin{claim} There is a constant $C(M^n)<\infty$ such that
\begin{equation}\label{spherecomp}
c_{p,\epsilon}(S^n)\leq C(M)c_{p,\epsilon}(M)
\end{equation}
for every $p\in (1,2)$ and $\epsilon>0$. 
\end{claim}
\begin{proof} We will construct a bounded linear map $\Phi: W^{1,p}(M,\mathbb{R}^2)\to W^{1,p}(S^n,\mathbb{R}^2)$ that fixes the constant maps and satisfies
$$E_{p,\epsilon}(\Phi(u))\leq C(M)E_{p,\epsilon}(u)$$
for all $u\in W^{1,p}(M,\mathbb{R}^2)$ and $p\in (1,2)$. For any family $h\in \Gamma_p(M)$, we then see that $\Phi\circ h$ defines a family in $\Gamma_p(S^n)$, so that
$$c_{p,\epsilon}(S^n)\leq \max_yE_{p,\epsilon}(\Phi(h_y))\leq C(M)\max_y E_{p,\epsilon}(h_y),$$
and taking the infimum over $h\in \Gamma_p(M)$ gives (\ref{spherecomp}).

\hspace{3mm} We construct this map $\Phi$ as follows. First, denote by $S^n_+$ the northern hemisphere $S^n_+=\{(x^1,\ldots,x^{n+1})\in S^n\mid x^{n+1}\geq 0\} $, and consider the reflection map
$$R: W^{1,p}(S^n_+,\mathbb{R}^2)\to W^{1,p}(S^n,\mathbb{R}^2)$$
given by
$$(Ru)(x^1,\ldots,x^{n+1})=u(x^1,\ldots,x^n,|x^{n+1}|).$$
$R$ is clearly a bounded linear map which fixes the constants, and has the effect of doubling $E_{p,\epsilon}$--i.e.,
$$E_{p,\epsilon}(Ru,S^n)=2E_{p,\epsilon}(u,S^n_+)$$
for every $u\in W^{1,p}(S^n_+)$.

\hspace{3mm} Next, since $S^n_+$ is a topological ball, we can choose some smooth $f:S^n_+\hookrightarrow M$ which is a diffeomorphism onto its image. Fixing such an $f$, we see that the pullback map
$$P_f:W^{1,p}(M,\mathbb{R}^2)\to W^{1,p}(S^n_+,\mathbb{R}^2)$$
given by
$$(P_fu)=u\circ f$$
is another bounded linear map that fixes the constant maps, and satisifies
$$E_{p,\epsilon}(P_fu,S^n_+)\leq C(M)E_{p,\epsilon}(u,M).$$
In particular, taking $\Phi:=R\circ P_f$ gives a map $\Phi: W^{1,p}(M,\mathbb{R}^2)\to W^{1,p}(S^n,\mathbb{R}^2)$ satisfying the desired properties, confirming the claim.
\end{proof}

\hspace{3mm} Finally, taking the supremum over $\epsilon>0$ in (\ref{spherecomp}), we see that
$$c_p(S^n)\leq C(M)c_p(M).$$
Combining this with (\ref{spherelbd}), it follows finally that
\begin{equation}
\liminf_{p\to 2}(2-p)c_p(M)\geq C(M)^{-1}\liminf_{p\to 2}(2-p)c_p(S^n)>0,
\end{equation} 
as desired. In particular, putting this together with the conclusion of Proposition \ref{exist1}, we arrive finally at the result of Theorem \ref{existthm}.

\section{Appendix}

\subsection{Proof of Proposition \ref{pharmfcnreg}} \hfill

\hspace{3mm} In this short section, we demonstrate the independence from the parameter $p\in [\frac{3}{2},2]$ of some standard estimates for $p$-harmonic functions (namely, the $W^{1,\infty}$ and $W^{2,p}$ estimates discussed in Proposition \ref{pharmfcnreg}). This is simply a matter of keeping track of $p$ in the estimates of \cite{DiB} and \cite{Lew}, but we give some details in the interest of completeness.

\hspace{3mm} Let $B_2(x)$ be a geodesic ball in a manifold $M^n$ satisfying the sectional curvature bound
$$|sec(M)|\leq k,$$
and let $\varphi\in W^{1,p}(B_2(x),\mathbb{R})$ be a $p$-harmonic function on $B_2(x)$ for $p\in [\frac{3}{2},2]$. Recall that, by the convexity of the $p$-energy functional, $\varphi$ must be the unique minimizer for the $p$-energy with respect to its Dirichlet data.

\hspace{3mm} For $\epsilon>0$, we consider as in \cite{Lew} the perturbed $p$-energy functionals
$$F_{\epsilon}(\psi)=\int (\epsilon+|d\psi|^2)^{p/2},$$
and let $\varphi_{\epsilon}\in W^{1,p}(B_2(x))$ minimize $F_{\epsilon}(\psi)$ with respect to the condition $\psi-\varphi \in W_0^{1,p}(B_2(x))$. Setting
$$\gamma_{\epsilon}:=(\epsilon+|d\varphi_{\epsilon}|^2)^{1/2},$$
we then have that
\begin{equation}\label{perteq}
div(\gamma_{\epsilon}^{p-2}d\varphi_{\epsilon})=0,
\end{equation}
and by standard results on quasilinear equations of this form (see, e.g., Chapter 4 of \cite{LU}), it follows that $\varphi_{\epsilon}$ is a smooth, classical solution of (\ref{perteq}). Moreover, since $\varphi$ is the unique $p$-energy minimizer with respect to its Dirichlet data, we know that $\varphi_{\epsilon}\to \varphi$ strongly in $W^{1,p}(B_2)$ as $\epsilon\to 0$. The task now (as in \cite{DiB}, \cite{Lew}) is to establish estimates of the form given in (\ref{pharmfcnreg}) for the perturbed solutions $\varphi_{\epsilon}$, and pass them to the limit $\epsilon\to 0$.

\hspace{3mm} As in \cite{Lew}, we observe now that, for $\varphi_{\epsilon}$ solving (\ref{perteq}), the energy density $\gamma_{\epsilon}^p$ satisfies the divergence-form equation
\begin{equation}
div(A_{\epsilon}\nabla (\gamma_{\epsilon}^p))=p\gamma_{\epsilon}^{p-2}[\langle A_{\epsilon},Hess(\varphi_{\epsilon})^2\rangle+Ric(d\varphi_{\epsilon},d\varphi_{\epsilon})],
\end{equation}
where $Hess(\varphi_{\epsilon})^2$ denotes the composition
$$Hess(\varphi_{\epsilon})^2(X,Y)=tr(Hess(\varphi_{\epsilon})(X,\cdot)Hess(\varphi_{\epsilon})(Y,\cdot)),$$
and
\begin{equation}\label{adef}
A_{\epsilon}:=I+(p-2)\gamma_{\epsilon}^{-2}d\varphi_{\epsilon}\otimes d\varphi_{\epsilon}.
\end{equation}
In particular, it follows that
\begin{equation}\label{bochy}
div(A_{\epsilon}\nabla (\gamma_{\epsilon}^p))\geq p(p-1)\gamma_{\epsilon}^{p-2}|Hess(\varphi_{\epsilon})|^2-C(n,k)\gamma_{\epsilon}^p.
\end{equation}

\hspace{3mm} Now, since $|\nabla \gamma_{\epsilon}|\leq |Hess(\varphi_{\epsilon})|,$ when we integrate (\ref{bochy}) against a test function $\psi\in C_c^{\infty}(B_2(x))$ with $\psi \equiv 1$ on $B_1(x)$ and $|\nabla \psi|\leq 2$, we find that
\begin{eqnarray*}
\int \psi^2 p(p-1)\gamma_{\epsilon}^{p-2}|Hess(\varphi_{\epsilon})|^2&\leq & \int 2\psi |d\psi|\gamma_{\epsilon}^{p-1}|\nabla \gamma_{\epsilon}|+C(k,n)\gamma_{\epsilon}^p\\
&\leq &\int_{B_2} 4\gamma_{\epsilon}^{p/2}(\psi \gamma_{\epsilon}^{\frac{p-2}{2}}|Hess(\varphi_{\epsilon})|)+C(k,n)\gamma_{\epsilon}^p,
\end{eqnarray*}
and an application of Young's inequality yields
\begin{equation}
p(p-1)\int \psi^2 \gamma_{\epsilon}^{p-2}|Hess(\varphi_{\epsilon})|^2\leq \frac{C(k,n)}{(p-1)}\int_{B_2}\gamma_{\epsilon}^p.
\end{equation}
In particular, since H\"{o}lder's inequality gives
$$\int_{B_1} |Hess(\varphi_{\epsilon})|^p\leq \left(\int_{B_1}\gamma_{\epsilon}^{p-2}|Hess(\varphi_{\epsilon})|^2\right)^{p/2}\left(\int\gamma_{\epsilon}^p\right)^{\frac{2-p}{2}},$$
it follows that
$$\|d\varphi_{\epsilon}\|^p_{W^{1,p}(B_1)}\leq \frac{C(k,n)}{(p-1)^2}\int_{B_2}\gamma_{\epsilon}^p,$$
and since $p\in [\frac{3}{2},2]$, we can rewrite this as
\begin{equation}\label{pertw2p}
\|d\varphi_{\epsilon}\|^p_{W^{1,p}(B_1)}\leq C(k,n)\int_{B_2}\gamma_{\epsilon}^p.
\end{equation}

\hspace{3mm} To obtain $L^{\infty}$ estimates for $\gamma_{\epsilon}$, we can apply Moser iteration to (\ref{bochy}). Since the eigenvalues of 
$$A_{\epsilon}=I+(p-2)\gamma_{\epsilon}^{-2}d\varphi_{\epsilon}\otimes d\varphi_{\epsilon}$$
are bounded between $p-1$ and $1$, and we are working with $p\in [\frac{3}{2},2]$, it is easy to see that the resulting estimate has the desired form
\begin{equation}\label{pertlip}
\|d\varphi_{\epsilon}\|_{L^{\infty}(B_1)}^p\leq \|\gamma_{\epsilon}\|_{L^{\infty}(B_1)}^p\leq C(k,n)\int_{B_2}\gamma_{\epsilon}^p.
\end{equation}

\hspace{3mm} Finally, since $\varphi_{\epsilon}\to \varphi$ strongly in $W^{1,p}(B_2(x))$, we have that
$$\lim_{\epsilon\to 0}\int_{B_2}\gamma_{\epsilon}^p=\int_{B_2} |d\varphi|^p,$$
and it follows from (\ref{pertlip}) and (\ref{pertw2p}) that
\begin{equation}
\|d\varphi\|^p_{L^{\infty}(B_1)}\leq \liminf_{\epsilon\to 0}\|d\varphi_{\epsilon}\|_{L^{\infty}(B_1)}^p\leq C(k,n)\int_{B_2}|d\varphi|^p,
\end{equation}
and
\begin{equation}
\|d\varphi\|_{W^{1,p}(B_1)}^p\leq \liminf_{\epsilon\to 0}\|d\varphi_{\epsilon}\|^p_{W^{1,p}(B_1)}\leq C(k,n)\int_{B_2}|d\varphi|^p.
\end{equation} 
Proposition \ref{pharmfcnreg} then follows by scaling.

\subsection{Proof of Lemma \ref{bigapplem}} \hfill

\hspace{3mm} In this section, we prove Lemma \ref{bigapplem}, which we employed in the proof of Corollary \ref{tuwbds}. For convenience, we restate the lemma here:

\begin{lem}\label{bigapplem2} Let $B_2(x)$ be a geodesic ball in a manifold $M^n$ of sectional curvature $|sec(M)|\leq k$ and injectivity radius $inj(M)\geq 3$. Let $S$ be an $(n-2)$-current in $W^{-1,p}(B_2(x))$ satisfying, for some constant $A$,
\begin{equation}
\langle S,\zeta\rangle \leq A r^{n-p}\|\zeta\|^{p-1}_{L^{\infty}}\|d\zeta\|^{2-p}_{L^{\infty}}\text{\hspace{3mm}}\forall \zeta\in \Omega_c^{n-2}(B_r(y))
\end{equation}
for every ball $B_r(y)\subset B_2(x)$. Suppose also that the $r$-tubular neighborhoods $\mathcal{N}_r(spt(S))$ about the support of $S$ satisfy
\begin{equation}
Vol(B_2(x)\cap \mathcal{N}_r(spt(S)))\leq Ar^p.
\end{equation}
Then there is a constant $C(n,k,A)$ such that for every $1<q<p$, we have
\begin{equation}
\|S\|_{W^{-1,q}(B_1(x))}\leq C(n,k,A)(p-q)^{-1/q}.
\end{equation}
\end{lem}

\hspace{3mm} Let's begin now by making some simple reductions. First, since the given metric $g$ on $B_2(x)$ is uniformly equivalent to the flat one $g_0$ with
$$C(n,k)^{-1}g_0\leq g\leq C(n,k)g_0$$
for some constant $C(n,k)$, it will suffice to establish the lemma in the flat case. Next, we note that every $(n-2)$-current $S$ in $B^n_2(0)\subset \mathbb{R}^n$ is described by a finite collection of scalar distributions $S_{ij}$, where
$$\langle S_{ij},\varphi\rangle:=\langle S, *\varphi dx^i\wedge dx^j\rangle.$$
Thus, it is enough to show that Lemma \ref{bigapplem2} holds with a scalar distribution $f$ in place of the $(n-2)$-current $S$. Our first step in proving this is then the following observation:

\begin{lem}\label{wgradestlem} For $p\in (1,2)$, let $f\in W^{-1,p}(B_2^n(0))$ be a distribution on $B_2^n(0)$ satisfying the estimate
\begin{equation}\label{fscalebds}
\langle f,\varphi\rangle\leq Ar^{n-p}\|\varphi\|_{L^{\infty}}^{p-1}\|d\varphi\|_{L^{\infty}}^{2-p}\text{\hspace{3mm} }\forall\varphi\in C_c^{\infty}(B_r(x))
\end{equation}
for every ball $B_r(x)\subset B_2^n$. Fixing a cutoff function $\chi\in C_c^{\infty}(B_{5/3}(0))$ such that $\chi\equiv 1$ on $B_{4/3}(0)$, set
$$w(x):=\langle (\chi f)(y),G(x-y)\rangle,$$
where $G$ is the $n$-dimensional Euclidean Green's function.
We then have for $x\in B_1(0)\setminus spt(f)$ a pointwise gradient estimate of the form
\begin{equation}\label{wiptwise}
|dw(x)|\leq C_n A\cdot dist(x,spt(f))^{-1}.
\end{equation}
\end{lem}

\begin{proof} For $x\in B_1\setminus spt(f)$, we observe that the pointwise derivatives $w_i(x):=\partial_iw(x)$ are well-defined, and given by
$$w_i(x):=\partial_iw(x)=c_n\langle (\chi f)(y), |x-y|^{-n}(x-y)_i\rangle,$$
where $c_n$ is a dimensional constant. 

\hspace{3mm} To establish (\ref{wiptwise}), first choose a function $\zeta\in C_c^{\infty}([\frac{1}{2},2])$ satisfying
$$\zeta\equiv 1\text{ on }[\frac{3}{4},\frac{3}{2}]\text{ and }|\zeta'|\leq 10,$$
and for $j\in \mathbb{Z}$, set 
$$\zeta_j(t):=\zeta(2^{-j}t).$$
Defining
$$\eta_j(t):=\frac{\zeta_j(t)}{\Sigma_{k\in\mathbb{Z}}\zeta_k(t)},$$
it's easy to see that the functions $\eta_j$ satisfy
\begin{equation}\label{etajprop1}
spt(\eta_j)\subset (2^{j-1},2^{j+1}),
\end{equation}
\begin{equation}
\Sigma_{j\in\mathbb{Z}}\eta_j(t)=1,
\end{equation}
and
\begin{equation}\label{etajprop3}
|\eta_j'|\leq 10\cdot 2^{-j}.
\end{equation}

\hspace{3mm} Given $x\in B_1^n\setminus spt(f)$, let $m=\lceil |\log_2\delta|\rceil$, so that
$$2^{1-m}\geq dist(x,spt(f))\geq 2^{-m}.$$
Writing
\begin{eqnarray*}
w_i(x)&=&c_n\langle (\chi f)(y),|x-y|^{-n}(x-y)_i\rangle\\
&=&c_n\langle (\chi f)(y),\Sigma_{j\in\mathbb{Z}}\eta_j(|x-y|)|x-y|^{-n}(x-y)_i\rangle,
\end{eqnarray*}
and observing that 
$$1-\Sigma_{j=-m}^2\eta_j(|x-y|)=0$$
when $y\in spt(\chi f)\subset B_4(x)\setminus B_{2^{-m}}(x),$ it follows that
\begin{eqnarray*}
w_i(x)&=&c_n\langle (\chi f)(y),\Sigma_{j=-m}^2\eta_j(|x-y|)|x-y|^{-n}(x-y)_i\rangle\\
&=&c_n\Sigma_{j=-m}^2\langle (\chi f)(y),\eta_j(|x-y|)|x-y|^{-n}(x-y)_i\rangle.
\end{eqnarray*}

\hspace{3mm} Setting
$$\varphi_j(y):=\chi(y)\eta_j(|x-y|)|x-y|^{-n}(x-y)_i,$$
we can then use (\ref{etajprop1})-(\ref{etajprop3}) to see that
$$spt(\varphi_j)\subset B_{2^{j+1}}(x),$$
$$\|\varphi_j\|_{L^{\infty}}\leq 2^{(j-1)(1-n)},$$
and 
$$\|d\varphi_j\|_{L^{\infty}}\leq C_n2^{-n(j-1)}.$$
By (\ref{fscalebds}), it therefore follows that
\begin{eqnarray*}
|\langle f,\varphi_j\rangle|&\leq & A(2^{j+1})^{n-p}\|\varphi_j\|_{L^{\infty}}^{p-1}\|d\varphi_j\|_{L^{\infty}}^{2-p}\\
&\leq & C_nA(2^{j+1})^{n-p}\cdot 2^{(j-1)(1-n)(p-1)}\cdot 2^{-n(2-p)(j-1)}\\
&\leq & C_n' A 2^{-j}.
\end{eqnarray*}
Summing from $j=-m$ to $j=2$, we obtain finally
\begin{eqnarray*}
|w_i(x)|&=&|c_n\Sigma_{j=-m}^2\langle f, \varphi_j(y)\rangle|\\
&\leq & C_nA\Sigma_{j=-m}^2 2^{-j}\\
&\leq & C_n A 2^m\\
&\leq & 2C_n A \cdot dist(x,spt(f))^{-1},
\end{eqnarray*}
giving the desired estimate (\ref{wiptwise}).
\end{proof}

\begin{cor}\label{wneg1qbds}
Let $f\in W^{-1,p}(B_2^n(0))$ be as in Lemma \ref{wgradestlem}, satisfying
\begin{equation}
\langle f,\varphi\rangle\leq Ar^{n-p}\|\varphi\|_{L^{\infty}}^{p-1}\|d\varphi\|_{L^{\infty}}^{2-p} \text{ }\forall \varphi \in C_c^{\infty}(B_r(x))
\end{equation}
for every ball $B_r(x)\subset B_2(0)$. In addition, suppose that the tubular neighborhoods $\mathcal{N}_r(spt(f))$ about the support of $f$ satisfy the volume bound
\begin{equation}\label{sptfvol}
Vol(\mathcal{N}_r(spt(f)))\leq Ar^p.
\end{equation}
Then there is a constant $C(n,A)<\infty$ depending only on $n$ and $A$ such that for every $q \in (1,p)$, we have the estimate
\begin{equation}
\|f\|_{W^{-1,q}(B_1(0))}\leq C(n,A)(p-q)^{-1/q}.
\end{equation}
\end{cor}

\begin{proof} By Lemma \ref{wgradestlem}, there exists a function $w\in W^{1,p}(B_2^n(0))$ satisfying
$$\Delta w=f\text{ on }B_{4/3}(0)$$
and
\begin{equation}\label{ptwisegradw}
|dw(x)|\leq \frac{C_nA}{dist(x,spt(f))}
\end{equation}
for $x\in B_1(0)\setminus spt(f)$. For any $\varphi \in C_c^{\infty}(B_1(0))$ and $q\in (1,p)$, we then have
\begin{eqnarray*}
\langle f,\varphi\rangle&=&\langle \Delta w, \varphi\rangle\\
&=&-\int \langle dw,d\varphi\rangle\\
&\leq &\|dw\|_{L^q(B_1(0))}\|d\varphi\|_{L^{q'}},
\end{eqnarray*}
while, by (\ref{ptwisegradw}) and (\ref{sptfvol}), we see that
\begin{eqnarray*}
\int_{B_1(0)}|dw|^q&\leq&C_nA^q\int_{B_1(0)}dist(x,spt(f))^{-q}\\
&\leq &CA^q\int_0^2q r^{-q-1}Vol(\mathcal{N}_r(spt(f)))dr\\
&\leq &C A^{q+1}\int_0^2 r^{p-q-1}dr\\
&\leq &\frac{C A^{q+1}}{p-q}.
\end{eqnarray*}
Thus, we indeed have
$$\langle f,\varphi\rangle \leq C(n,A)(p-q)^{-1/q}\|d\varphi\|_{L^{q'}},$$
the desired $W^{-1,q}$ estimate.
\end{proof}

As remarked previously, Lemma \ref{bigapplem2} now follows by applying Corollary \ref{wneg1qbds} to the scalar component distributions of the $(n-2)$-current $S$.


\begin{thebibliography}{199}

\bibitem{ABO} G. Alberti, S. Baldo, G. Orlandi, \emph{Functions with prescribed singularities,} J. Eur. Math. Soc. 5 (2003), 275-311.

\bibitem{All} W.K. Allard, \emph{On the first variation of a varifold,} Ann. of Math. 95 (1972), 417-491.

\bibitem{AS} L. Ambrosio and H. Soner, \emph{A measure theoretic approach to higher codimension mean curvature flow,} Ann. Sc. Norm. Sup. Pisa, Cl. Sci. (4) 25 (1997), 27–49.

\bibitem{BBH} F. Bethuel, H. Brezis, and F. H\'{e}lein, \emph{Ginzburg-Landau vortices}, Progress in Nonlinear Diﬀerential Equations and their Applications, vol. 13, Birkh\"{a}user, Boston (1994)

\bibitem{BBO} F. Bethuel, H. Brezis, and G. Orlandi, \emph{Asymptotics for the Ginzburg-Landau equation in arbitrary dimensions}, J. Funct. Anal. 186 (2001), 432-520.

\bibitem{BOSp} F. Bethuel, G. Orlandi, and D. Smets, \emph{Convergence of the parabolic Ginzburg-Landau equation to motion by mean curvature}, Annals of Mathematics, Vol. 163, No. 1 (2006), 37-163.

\bibitem{BMR} H. Brezis, F. Merle, T. Rivi\`{e}re, \emph{Quantization effects for $-\Delta u=u(1-|u|^2)$ in $\mathbb{R}^2$,} Arch. Rational Mech. Anal. 126 (1994) 35–58.

\bibitem{CH} B. Chen, R. Hardt, \emph{Prescribing singularities for p-harmonic mappings,} Indiana Univ. Math. J., 44
(1995), 575-601.

\bibitem{CL} Y.M. Chen, F.H. Lin, \emph{Evolution of harmonic maps with Dirichlet
boundary conditions,} Comm. Anal. Geom., 1(3-4), (1993), 327-346.

\bibitem{CS} Y.M. Chen, M. Struwe, \emph{Existence and partial regularity results for the heat flow for harmonic maps,} Math. Z. 201(1), (1989), 83–103.

\bibitem{Cheng} D.R. Cheng, \emph{Asymptotics for the Ginzburg-Landau equation on manifolds with boundary under homogeneous Neumann condition}, arXiv preprint arXiv:1801.03987, (2018).

\bibitem{CM} M. Comte, P. Mironescu, \emph{Remarks on nonminimizing solutions of a Ginzburg-Landau type equation.}, Asymptot Anal. {\bf 13}(2) (1996), 199–215.

\bibitem{Dem} F. Demengel, \emph{Une caract\'{e}risation des applications de $W^{1,p}(B^n,S^1)$ qui peuvent $\hat{e}$tre approch\'{e}es par des fonctions r\'{e}guli\`{e}res,} C. R. Acad. Sci. Paris Sér. I Math. 310 (1990), no. 7, 553-557.

\bibitem{DiB} E. DiBenedetto, \emph{$C^{1+\alpha}$ local regularity of weak solutions of degenerate elliptic equations,} Nonlinear Analysis, vol. 7, issue 8 (1983), 827-850.

\bibitem{Gh} N. Ghoussoub, \emph{Duality and perturbation methods in critical point theory}, vol. 107, Cambridge University Press, 1993.

\bibitem{Gu} M. Guaraco, \emph{Min-max for phase transitions and the existence of embedded minimal hypersurfaces}, J. Diff. Geom., Vol. 108 No. 1, (2018), 91-133.

\bibitem{HL1} R. Hardt, F.H. Lin, \emph{Mappings minimizing the $L^p$ norm of the gradient,} Comm. Pure Appl. Math. 40 (1987), 556–588.

\bibitem{HL2} R. Hardt, F.H. Lin, \emph{Singularities for $p$-energy minimizing unit vector fields on planar domains,} Calc. Var. 3, (1995), 311-341.

\bibitem{JS1} R. L. Jerrard, H. M. Soner, \emph{The Jacobian and the Ginzburg-Landau energy,} Calc. Var. Partial Differential Equations 14 (2002) 151-191.

\bibitem{JS2} R. L. Jerrard, H. M. Soner, \emph{Functions of bounded higher variation,} Indiana Univ. Math. J., Vol. 51, No. 3, 645-677 (2002).

\bibitem{LU} O.A. Ladyzhenskaya, N.N. Uraltseva, \emph{Linear and Quasilinear Elliptic Equations}, Academic Press, New York, (1968).

\bibitem{Lew} J. Lewis, \emph{Regularity of the derivatives of solutions to certain degenerate elliptic equations,} Indiana Univ. Math. J., Vol. 32, No. 6, 849-858 (1983).

\bibitem{Lin} F.H. Lin, \emph{Gradient estimates and blow-up analysis for stationary harmonic maps,} Ann. of Math. (2)149, (1999), 785-829.

\bibitem{LR} F.-H. Lin, T. Rivi\`{e}re, \emph{Complex Ginzburg-Landau equations in high dimensions and codimension two area minimizing currents}, J. Eur. Math. Soc. {\bf 1}(3) (1999), 237-311. 

\bibitem{LW} F.H. Lin, C.Y. Wang, \emph{Harmonic and quasi-harmonic spheres,} Comm. Anal. Geom. 7(2) (1999), 397-429.

\bibitem{Luck} S. Luckhaus, \emph{Partial H\"{o}lder continuity for minima of certain energies among maps into a Riemannian manifold,} Ind. Univ. Math. J. 37, (1988), 349-367.

\bibitem{Mir} P. Mironescu, \emph{$S^1$-valued Sobolev maps,} Journal of Mathematical Sciences, Vol. 170, No. 3, (2010).

\bibitem{NVV} A. Naber, D. Valtorta, G. Veronelli, \emph{Quantitative regularity for $p$-harmonic maps,} arXiv preprint arXiv:1409.8537, (2014).

\bibitem{Rab} P.H. Rabinowitz, \emph{Minimax methods in critical point theory with applications to differential equations,} CBMS Reg. Conf. Ser. in Math. {\bf 65}, AMS (1986)

\bibitem{Simon} L. Simon, \emph{Lectures on geometric measure theory,} Proc. C.M.A. 3, Australian Nat. U. (1983) 

\bibitem{Scott} C. Scott, \emph{$L^p$ theory of differential forms on manifolds}, Transactions of the American Mathematical Society, Vol. 347, No. 6 (Jun., 1995), pp. 2075-2096 

\bibitem{Ste1} D.L. Stern, \emph{A natural min-max construction for Ginzburg-Landau functionals}, arXiv preprint arXiv:1612.00544, (2016).

\bibitem{Ste2} D.L. Stern, \emph{Energy concentration for min-max solutions of the Ginzburg-Landau equations on manifolds with $b_1(M)\neq 0$}, arXiv preprint arXiv:1704.0071, (2017)

\bibitem{TW} T. Toro, C. Wang, \emph{Compactness properties of weakly $p$-harmonic maps into homogeneous spaces,} Indiana Univ. Math. J., Vol. 44, No. 1 (1995), 87-113 .

\bibitem{Wang} C.Y. Wang, \emph{Limits of solutions to the generalized Ginzburg-Landau functional,} Comm. Partial Diff. Equations 27 (2002) 877–906.


\end{thebibliography}
\end{document}